\documentclass[twoside,a4paper,reqno,11pt]{amsart} 
\usepackage[top=28mm,right=28mm,bottom=28mm,left=28mm]{geometry}
\usepackage{microtype}
\usepackage{amsfonts, amsmath, amssymb, amsbsy, mathrsfs, bm, latexsym, stmaryrd, hyperref, array, enumitem, textcomp, url}
\usepackage[table]{xcolor}

\renewcommand{\a}{\alpha}
\renewcommand{\b}{\beta}

\newcommand{\e}{\epsilon}
\renewcommand{\l}{\lambda}
 
\newcommand{\s}{\sigma}
\renewcommand{\O}{\Omega}

\newcommand{\la}{\langle}
\newcommand{\ra}{\rangle}

\renewcommand{\L}{\Lambda}
\newcommand{\leqs}{\leqslant}
\newcommand{\geqs}{\geqslant}

\newcommand{\normeq}{\trianglelefteqslant}

\newcommand{\vs}{\vspace{3mm}}

\makeatletter
\newcommand{\imod}[1]{\allowbreak\mkern4mu({\operator@font mod}\,\,#1)}
\makeatother

\renewcommand{\geq}{\geqs}

\theoremstyle{plain}
\newtheorem{theorem}{Theorem} 
 
\newtheorem{corol}[theorem]{Corollary}
\newtheorem{thm}{Theorem}[section] 
\newtheorem{lem}[thm]{Lemma}
\newtheorem{prop}[thm]{Proposition} 
\newtheorem{con}[thm]{Conjecture}
\newtheorem{cor}[thm]{Corollary} 
\newtheorem{prob}[thm]{Problem} 
\newtheorem*{theorem*}{Theorem} 
\newtheorem*{conj*}{Conjecture}

\theoremstyle{definition}
\newtheorem{rem}[thm]{Remark}

\newtheorem{ex}[thm]{Example}
\newtheorem{remk}{Remark}
\newtheorem*{deff}{Definition}

\begin{document}

\title{On the soluble graph of a finite group}

\author{Timothy C. Burness}
\address{T.C. Burness, School of Mathematics, University of Bristol, Bristol BS8 1UG, UK}
\email{t.burness@bristol.ac.uk}

\author{Andrea Lucchini}
\address{A. Lucchini, Universit\`a di Padova, Dipartimento di Matematica ``Tullio Levi-Civita", Via Trieste 63, 35121 Padova, Italy}
\email{lucchini@math.unipd.it}

\author{Daniele Nemmi}
\address{D. Nemmi, Universit\`a di Padova, Dipartimento di Matematica ``Tullio Levi-Civita", Via Trieste 63, 35121 Padova, Italy}
\email{dnemmi@math.unipd.it}

\date{\today} 

\begin{abstract}
Let $G$ be a finite insoluble group with soluble radical $R(G)$. In this paper we investigate the soluble graph of $G$, which is a natural generalisation of the widely studied commuting graph. Here the vertices are the elements in $G \setminus R(G)$, with $x$ adjacent to $y$ if they generate a soluble subgroup of $G$. Our main result states that this graph is always connected and its diameter, denoted $\delta_{\mathcal{S}}(G)$, is at most $5$. More precisely,  we show that $\delta_{\mathcal{S}}(G) \leqs 3$ if $G$ is not almost simple and we obtain stronger bounds for various families of almost simple groups. For example, we will show that 
$\delta_{\mathcal{S}}(S_n) = 3$ for all $n \geqslant 6$. We also establish the existence of simple groups with $\delta_{\mathcal{S}}(G) \geqs 4$. For instance, we prove that 
$\delta_{\mathcal{S}}(A_{2p+1}) \geqslant 4$ for every Sophie Germain prime $p \geqslant 5$, which demonstrates that our general upper bound of $5$ is close to best possible. We conclude by briefly discussing some variations of the soluble graph construction and we present several open problems.
\end{abstract}

\maketitle

\section{Introduction}\label{s:intro}

Let $G$ be a non-abelian finite group and let $\Gamma(G)$ be the \emph{commuting graph} of $G$. Recall that the vertices of this graph are the non-central elements of $G$ and two distinct vertices are adjacent if and only if they commute in $G$. Commuting graphs arise naturally in many different contexts and they have been intensively studied by various authors in recent years.  

For example, Segev and Seitz \cite{SS} studied the connectivity properties of the commuting graphs of finite simple groups, which turned out to be a key ingredient in their work on the Margulis-Platonov conjecture on the normal subgroup structure of simple algebraic groups defined over number fields. Specifically, \cite[Theorem 8]{SS} shows that if $G$ is a finite simple classical group over a field of order greater than $5$, then either $\Gamma(G)$ is disconnected (and all such groups are determined), or the diameter of $\Gamma(G)$ is at least $4$ and at most $10$. A similar result for the commuting graphs of symmetric and alternating groups was established by Iranmanesh and Jafarzadeh in \cite{IJ}. Here they prove that if $G = S_n$ or $A_n$ and $\Gamma(G)$ is connected, then the diameter of $\Gamma(G)$ is at most $5$. More generally, they conjectured that if the commuting graph of any finite group is connected, then its diameter is bounded above by an absolute constant. This conjecture was subsequently refuted by Giudici and Parker in \cite{GP}, where an infinite sequence $(G_n)$ of $2$-groups of nilpotency class $2$ is constructed with the property that the diameter of $\Gamma(G_n)$ tends to infinity. However, if $G$ has trivial centre then a theorem of Morgan and Parker \cite{MP} states that the diameter of each connected component of $\Gamma(G)$ is at most $10$ (in \cite{Beike}, this has recently been extended to groups with $G'\cap Z(G)=1$). Although it is still not known whether or not an upper bound of $10$ is optimal, it is worth noting that Parker \cite{P} has proved that $\Gamma(G)$ has diameter at most $8$ if $G$ is soluble with trivial centre and he has shown that this bound is best possible in this setting.

Here we view the commuting graph of a group through a different lens, which leads naturally to some interesting generalisations that form the main focus of this paper.  To do this, first let $\mathcal{A}$ be the class of abelian groups and let $\Lambda_{\mathcal{A}}(G)$ be the graph with vertex set $G$ so that $x$ and $y$ are adjacent if and only if the subgroup $\la x, y \ra$ of $G$ is abelian. Clearly, every vertex in the centre $Z(G)$ is connected to every other vertex in this graph, so it makes sense to consider the more restrictive graph $\Gamma_{\mathcal{A}}(G)$, which is only defined on the  non-central elements of $G$. Then $\Gamma_{\mathcal{A}}(G)$ is the commuting graph of $G$ as defined above. Note that in this setting, $Z(G)$ is precisely  the set of isolated vertices in the complement of $\Lambda_{\mathcal{A}}(G)$.

Motivated by this observation, let $\mathcal{S}$ be the class of soluble groups and define the graph $\Lambda_{\mathcal{S}}(G)$ with vertices $G$ so that $x$ and $y$ are adjacent if and only if $\la x, y \ra$ is soluble. By a theorem of Guralnick et al. \cite[Theorem 1.1]{gu}, the isolated vertices in the complement coincide with the soluble radical $R(G)$ of $G$ and this leads us to the following definition.

\begin{deff}
Let $G$ be a finite insoluble group with soluble radical $R(G)$. The \emph{soluble graph} of $G$, denoted $\Gamma_{\mathcal{S}}(G)$, has vertex set $G \setminus R(G)$, with distinct vertices $x$ and $y$ adjacent if and only if $\la x,y \ra$ is a soluble subgroup of $G$.
\end{deff}

Our main aim in this paper is to investigate the connectivity properties of this graph as we range over the insoluble finite groups. Let $\delta_{\mathcal{S}}(G)$ be the diameter of $\Gamma_{\mathcal{S}}(G)$ (by convention, we set $\delta_{\mathcal{S}}(G)=\infty$ if $\Gamma_{\mathcal{S}}(G)$ is disconnected). Note that 
$\delta_{\mathcal{S}}(G) \geqs 2$ as a consequence of J.G. Thompson's celebrated classification of $N$-groups \cite{Th}, which implies that a finite group is soluble if and only if every $2$-generated subgroup is soluble (see Flavell \cite{Flav} for a direct proof).

A simplified version of our main result is the following. 

\begin{theorem}\label{t:main1}
Let $G$ be a finite insoluble group. Then $\Gamma_{\mathcal{S}}(G)$ is connected and
$\delta_{\mathcal{S}}(G) \leqs 5$. 
\end{theorem}

This is an immediate corollary of the more detailed statement in Theorem \ref{t:main2} below. In part (ii)(c), we refer to the following collections of simple groups:
\begin{align*}
\mathcal{A} & = \{ A_{11}, A_{12}, {\rm L}_{5}^{\e}(2),  {\rm M}_{12}, {\rm M}_{22}, {\rm M}_{23}, {\rm M}_{24}, {\rm HS}, {\rm J}_{3} \} \\
\mathcal{B} & = \{A_{n}, {\rm L}_7^{\e}(2), E_6(2), {\rm Co}_2, {\rm Co}_3, {\rm McL}, \mathbb{B}\}
\end{align*}
where $n \in \{19, 20, 23,24,31,43,44,47,48,59,60\}$.
Recall that a finite group is \emph{almost simple} if there exists a non-abelian simple group $G_0$ (the socle of $G$) such that $G_0 \normeq G \leqs {\rm Aut}(G_0)$.

\begin{theorem}\label{t:main2}
Let $G$ be a finite insoluble group. 
\begin{itemize}\addtolength{\itemsep}{0.2\baselineskip}
\item[{\rm (i)}] If $G$ is not almost simple, then $\delta_{\mathcal{S}}(G) \leqs 3$. 
\item[{\rm (ii)}] If $G$ is almost simple with socle $G_0$, then $\delta_{\mathcal{S}}(G) \leqs 5$. In addition: 

\vspace{1mm}

\begin{itemize}\addtolength{\itemsep}{0.2\baselineskip}
\item[{\rm (a)}] If $G_0 = {\rm L}_{2}(q)$ and $q \geqs 8$, then $\delta_{\mathcal{S}}(G) = 2$ if ${\rm PGL}_{2}(q) \leqs G$, otherwise $\delta_{\mathcal{S}}(G) = 3$.
\item[{\rm (b)}] If $G_0 = A_n$ and $n \geqs 7$, then either $\delta_{\mathcal{S}}(G) = 3$, or $G = A_n$ and $n \in \{p,p+1\}$, where $p$ is a prime with $p \equiv 3 \imod{4}$.
\item[{\rm (c)}] If $G \in \mathcal{A} \cup \mathcal{B}$ then $\delta_{\mathcal{S}}(G) \geqs 4$, with equality if $G \in \mathcal{A}$.
\item[{\rm (d)}] If $G = G_0$ is not isomorphic to a classical group, then 
$\delta_{\mathcal{S}}(G) \geqs 3$.
\end{itemize}
\end{itemize}
\end{theorem}

\begin{remk}\label{r:00}
As far as we are aware, the soluble graph of a finite group was first studied by Bhowal et al. in \cite{Bhowal}, which includes results on the genus, girth and clique number of $\Gamma_{\mathcal{S}}(G)$. This initial investigation was extended by Akbari et al. in \cite{Akbari}, where their main result states that $\Gamma_{\mathcal{S}}(G)$ is connected and $\delta_{\mathcal{S}}(G) \leqs 11$ (they also remark that they do not know an example with $\delta_{\mathcal{S}}(G) \geqs 4$). Therefore, Theorem \ref{t:main2} provides a significant strengthening of this earlier work on the diameter of $\Gamma_{\mathcal{S}}(G)$.
\end{remk}

In view of part (ii)(a) in Theorem \ref{t:main2}, we immediately obtain the following corollary.

\begin{corol}\label{c:1}
There are infinitely many finite simple groups $G$ with $\delta_{\mathcal{S}}(G)= 2$.
\end{corol}

\begin{remk}\label{r:1}
Let $G$ be a finite simple group with $\delta_{\mathcal{S}}(G)= 2$, so $G$ is a classical group by part (ii)(d) in Theorem \ref{t:main2}. Up to isomorphism, the only known examples are the $2$-dimensional linear groups ${\rm L}_2(q)$ with $q \geqs 4$ even (or $q=7$) and the unitary group ${\rm U}_4(2)$; it would be interesting to completely determine the simple groups with this property and we refer the reader to the end of Section \ref{ss:class} for some  additional results in this direction. For example, if $G = {\rm L}_n(q)$ then Proposition \ref{p:class_new} states that $\delta_{\mathcal{S}}(G) = 2$ if and only if $G$ is isomorphic to ${\rm L}_2(q)$ with $q \geqs 4$ even or $q=7$.
\end{remk}

\begin{remk}\label{r:0}
Part (ii)(c) in Theorem \ref{t:main2} establishes the existence of groups with $\delta_{\mathcal{S}}(G) \geqs 4$. In particular, it is worth noting that the Mathieu group $G = {\rm M}_{12}$ is the smallest finite group with $\delta_{\mathcal{S}}(G) \geqs 4$ (see Corollary \ref{c:m12}). However, we do not know if there is a group with $\delta_{\mathcal{S}}(G) = 5$ and thus determining the sharpness of our upper bound in Theorem \ref{t:main1} remains an open problem. On the other hand, it is straightforward to show that the upper bound in part (i) of Theorem \ref{t:main2} is best possible. Indeed, there are infinitely many groups $G$ that are not almost simple with $\delta_{\mathcal{S}}(G) = 3$ (see Corollary \ref{c:nas}).
\end{remk}

Recall that a prime $p$ is a \emph{Sophie Germain prime} if $2p+1$ is also a prime number; the examples with $p<200$ are as follows:
\[
\{ 2, 3, 5, 11, 23, 29, 41, 53, 83, 89, 113, 131, 173, 179, 191\}.
\]
It is conjectured that there are infinitely many such primes, but this remains a formidable open problem in number theory. Modulo this conjecture, our next result establishes 
the existence of infinitely many simple groups with $\delta_{\mathcal{S}}(G) \geqs 4$ (see Theorem \ref{t:sophie}).

\begin{theorem}\label{t:main4}
If $p \geqs 5$ is a Sophie Germain prime, then $\delta_{\mathcal{S}}(A_{2p+1}) \in \{4,5\}$.
\end{theorem}

Suppose $R(G)=1$ and observe that the set of involutions in $G$ forms a clique in $\Gamma_{\mathcal{S}}(G)$ since any two involutions generate a dihedral group. Therefore, the bound $\delta_{\mathcal{S}}(G) \leqs 5$ in Theorem \ref{t:main1} will follow if we can show that for all nontrivial $x \in G$ there is a path in $\Gamma_{\mathcal{S}}(G)$ of length at most $2$ from $x$ to an involution. For instance, such a path exists if we can find an element $y \in N_G(\la x \ra)$ such that $|N_G(\la y \ra)|$ is even. With this observation in mind, our proof of Theorem \ref{t:main2} will establish the following result (in the statement, $\delta(x,y)$ denotes the distance in $\Gamma_{\mathcal{S}}(G)$ from $x$ to $y$). 

\begin{theorem}\label{t:main3}
Let $G$ be a finite insoluble group with $R(G)=1$ and let $x \in G$ be nontrivial. Then either
\begin{itemize}\addtolength{\itemsep}{0.2\baselineskip}
\item[{\rm (i)}] There exists an involution $y \in G$ with $\delta(x,y) \leqs 2$; or
\item[{\rm (ii)}] $G$ is the Mathieu group ${\rm M}_{23}$ and $|x|=23$.
\end{itemize}
\end{theorem}

\begin{remk}\label{r:m23}
The special case in part (ii) of Theorem \ref{t:main3} is a genuine exception. Indeed, suppose $G = {\rm M}_{23}$ and $x \in G$ has order $23$. Let $B_{\ell}(x)$ denote the ball of radius $\ell$ in $\Gamma_{\mathcal{S}}(G)$ centred at $x$. Since $H = N_G(\la x \ra) = 23{:}11$ is the unique maximal subgroup of $G$ containing $x$, it follows that $B_1(x)$ is the set of nontrivial elements in $H$. Suppose $y \in B_1(x)$ has order $11$ and let $J$ be a maximal subgroup of $G$ containing $y$. Then either $J = 23{:}11$ is the normaliser of a Sylow $23$-subgroup of $G$, or $J$ is a unique copy of ${\rm M}_{11}$ or ${\rm M}_{22}$. In the latter two cases, $N_J(\la y \ra) = 11{:}5$ is the unique maximal soluble subgroup of $J$ containing $y$ and this allows us to conclude that every element in $B_2(x)$ has order $5$, $11$ or $23$. Finally, if $z \in G$ has order $5$ then $|N_G(\la z \ra)|$ is even and thus the shortest path from $x$ to an involution has length $3$.
\end{remk}

\begin{remk}\label{r:lit}
As in Theorem \ref{t:main3}, let $G$ be a finite insoluble group with $R(G)=1$ and let $x \in G$ be nontrivial. In the proof of \cite[Theorem 4.2]{Akbari}, it is shown that there exists an involution $y \in G$ with $\delta(x,y) \leqs 5$, which is the key step in establishing the connectivity of $\Gamma_{\mathcal{S}}(G)$ and the bound $\delta_{\mathcal{S}}(G) \leqs 11$ in the main theorem of \cite{Akbari}.
\end{remk}

\begin{remk}\label{r:lit2}
Let us also highlight a connection between Theorem \ref{t:main3} and earlier work of Hagie \cite{Hagie}. Let $G$ be a finite group and let $\pi(G)$ be the set of prime divisors of $|G|$. Consider the graph with vertex set $\pi(G)$, where distinct vertices $p$ and $q$ are joined by an edge if $G$ has a soluble subgroup of order divisible by $pq$ (note that this is a natural generalisation of the widely studied \emph{prime graph} of $G$, where ``soluble" is replaced by ``cyclic"). Let $d(p,q)$ be the distance between two vertices $p$ and $q$ in this graph. Then \cite[Theorem 2]{Hagie} states that if $G$ has even order, then $d(2,p) \leqs 3$ for all $p \in \pi(G)$. As an immediate corollary of Theorem \ref{t:main3}, we can strengthen this result as follows: if $p \in \pi(G)$, then either $d(2,p) \leqs 2$, or $G = {\rm M}_{23}$, $p=23$ and $d(2,p) = 3$.
\end{remk}

In order to state our next result, let $P_n$ be the path graph with $n$ vertices and recall that a graph $\Gamma$ is a \emph{cograph} if it has no induced subgraph isomorphic to the four-vertex path $P_4.$ There are several equivalent characterisations of this property. For instance, $\Gamma$ is a cograph if any of the following conditions hold:
\begin{itemize}\addtolength{\itemsep}{0.2\baselineskip}
\item[{\rm (a)}] Every  induced subgraph of $\Gamma$ has the property that any maximal clique intersects any maximal independent set in a single vertex.
\item[{\rm (b)}] Every induced subgraph of $\Gamma$ with more than one vertex has at least two vertices with equal neighbourhoods.
\item[{\rm (c)}] Every connected induced subgraph of $\Gamma$ has a disconnected complement. 
\end{itemize}
The complement of the soluble graph $\Gamma_{\mathcal{S}}(G)$ is the \emph{insoluble graph} of $G$: the vertices are once again labelled by the elements of $G\setminus R(G)$, with $x$ and $y$ adjacent if they generate an insoluble group. By \cite[Theorem 6.4]{gu}, the insoluble graph of a finite insoluble group is connected with diameter $2$ and as an immediate consequence (see (c) above) we obtain the following result.

\begin{theorem}\label{t:cog}
Let $G$ be a finite insoluble group. Then $\Gamma_{\mathcal{S}}(G)$ is not a cograph.
\end{theorem}

Next let us recall the notion of a \emph{dual pair} of graphs, which was recently introduced by Cameron in \cite[Section 12.1]{Cam}. Let $B$ be a bipartite graph with parts $V_1$ and $V_2$.  The corresponding \emph{halved graphs} arising from $B$ are the graphs $\Gamma_1$ and $\Gamma_2$ with respective vertex sets $V_1$ and $V_2$, where two vertices are adjacent in the relevant graph if and only if they lie at distance $2$ in $B$. We then say that a given pair of graphs $\Gamma_1$ and $\Gamma_2$ is a \emph{dual pair} if there is a bipartite graph $B$ without isolated vertices such that $\Gamma_1$ and $\Gamma_2$ are the halved graphs of $B$. In this situation, $\Gamma_1$ is connected if and only if $\Gamma_2$ is connected. More generally, there is a natural bijection between the connected components of $\Gamma_1$ and $\Gamma_2$  with the property that the diameters of the corresponding components are either equal or differ by $1$ (see \cite[Proposition 12.1]{Cam}).

With this definition in hand, let $G$ be a finite insoluble group and define the \emph{soluble intersection graph} ${\rm Int}_{\mathcal{S}}(G)$ of $G$ to be the graph
whose vertices are the nontrivial soluble subgroups of $G$, where two vertices $H$ and $K$ are adjacent if and only if $H \cap  K \ne 1$. If $R(G)\neq 1,$ then $R(G)$ is a universal vertex of ${\rm Int}_{\mathcal{S}}(G)$. On the other hand, if $R(G)=1$ then 
the graphs ${\rm Int}_{\mathcal{S}}(G)$ and $\Gamma_{\mathcal{S}}(G)$ form a dual pair by \cite[Proposition 12.2]{Cam}. Therefore, we obtain the following corollary as an immediate consequence of the bound in Theorem \ref{t:main1}.

\begin{corol}\label{c:int}
Let $G$ be a finite insoluble group. Then ${\rm Int}_{\mathcal{S}}(G)$ is connected, with diameter at most $6$.
\end{corol}

\begin{remk}\label{r:ng}
Let $G$ be a finite group and recall that the \emph{non-generating graph} of $G$ is a graph on the nontrivial elements of $G$, where $x$ is adjacent to $y$ if and only if $G \ne \la x,y \ra$. Also recall that the vertices of the \emph{intersection graph} of $G$ are the nontrivial proper subgroups of $G$, with $H$ and $K$ adjacent if and only if $H \cap K \ne 1$. It is straightforward to show that these two graphs form a dual pair as defined above. In recent work, Freedman \cite{Free} has proved that the intersection graph of a finite non-abelian simple group is connected with diameter at most $5$ (this bound is tight for the Baby Monster sporadic group, for example), whence the non-generating graph of a non-abelian finite simple group has diameter at most $6$ by \cite[Proposition 12.1]{Cam}. In fact, \cite[Theorem 6.5.4]{Free_thesis} shows that the latter diameter is at most $4$. Now, if $G$ is a finite insoluble group with $R(G)=1$, then the soluble graph $\Gamma_{\mathcal{S}}(G)$ is a subgraph of the non-generating graph and therefore it is natural to compare the diameters of these two connected graphs. In Corollary \ref{c:nongen} we prove that there exist finite simple groups $G$ such that $\delta_{\mathcal{S}}(G)$ is strictly larger than the diameter of the non-generating graph of $G$ (indeed, we can take $G = A_{2p+1}$ for any Sophie Germain prime $p \geqs 5$).  
\end{remk}

We close with some comments on the organisation of this paper. In Section \ref{s:prel} we present some preliminary results that will be needed in the proofs of our main theorems. This includes Lemma \ref{l:rad}, which provides an immediate reduction to the groups with trivial soluble radical, and Lemma \ref{l:base} on a useful connection with bases for primitive permutation groups. In Section \ref{ss:comp} we discuss some of our main computational methods (for example, we use {\sc Magma} \cite{magma} extensively in studying the soluble graphs of the almost simple sporadic groups). The main result in Section \ref{s:red} is Theorem \ref{t:nas}, which establishes the bound $\delta_{\mathcal{S}}(G) \leqs 3$ whenever $G$ is an insoluble group that is not almost simple. This allows us to focus our attention on the almost simple groups for the remainder of the paper and we consider the various possibilities in turn. We begin in Section \ref{s:spor} by handling the groups with socle a sporadic simple group (the main results are Theorems \ref{t:spor} and \ref{t:spor2}). Next we turn to the symmetric and alternating groups in Section \ref{s:alt}, establishing our main results in Theorems \ref{t:sym}, \ref{t:alt} and \ref{t:sophie}. Finally, the almost simple groups of Lie type are the focus of Section \ref{s:lie}. Here the linear groups with socle ${\rm L}_2(q)$ merit special attention and they are handled separately in Section \ref{ss:psl2}. The exceptional groups are treated in Section \ref{ss:excep}, followed by the remaining classical groups in Section \ref{ss:class}. The main results in Section \ref{s:lie} are Theorem \ref{t:psl2} and Corollaries \ref{c:ex} and \ref{c:class}. We conclude with a brief discussion of some natural generalisations of the soluble graph in Section \ref{s:gen} and we present a number of open problems in Section \ref{s:op}. 

\vs

\noindent \textbf{Acknowledgements.} We sincerely thank an anonymous referee for their very careful reading of an earlier version, which has helped us to improve the accuracy and presentation of the paper. Burness thanks the Department of Mathematics at the University of Padua for their generous hospitality during a research visit in autumn 2021.

\section{Preliminaries}\label{s:prel}

In this section we record some preliminary results, which will be useful in the proofs of our main theorems.

\subsection{The soluble graph}\label{ss:sg}

Let $G$ be a finite insoluble group and let $\Gamma_{\mathcal{S}}(G)$ be the soluble graph of $G$. As defined above, the vertices of this graph are the elements in $G \setminus R(G)$, where $R(G)$ is the soluble radical of $G$, and two distinct vertices $x$ and $y$ are adjacent (denoted $x \sim y$) if and only if $\la x,y \ra$ is soluble. Let $\delta(x,y)$ denote the distance between the two vertices $x$ and $y$ in this graph (if $x$ and $y$ are not connected by a path, then set $\delta(x,y) = \infty$). Then 
\[
\delta_{\mathcal{S}}(G) = \max\{ \delta(x,y) \,:\, x, y \in G \setminus R(G)\}
\]
is the diameter of $\Gamma_{\mathcal{S}}(G)$. For $x \in G \setminus R(G)$ and $\ell \in \mathbb{N}$ we also define
\[
B_{\ell}(x) = \{ y \in G \setminus R(G) \, :\, \delta(x,y) \leqs \ell \},
\]
which is the ball of radius $\ell$ in $\Gamma_{\mathcal{S}}(G)$, centred at $x$. 

\begin{rem}\label{r:useful}
Note that $\delta_{\mathcal{S}}(G) \geqs 3$ if and only if there exist elements $x,y \in G \setminus R(G)$ such that $B_1(x) \cap B_1(y)$ is empty. Also observe that $\delta_{\mathcal{S}}(G) \geqs 4$ if there exist $x,y \in G \setminus R(G)$ such that $\la a,b \ra$ is insoluble for all $a \in B_1(x)$, $b \in B_1(y)$.
\end{rem}

We begin by recording some preliminary observations. Since $R(G/R(G))=1$, the following elementary result allows us to focus our attention on the groups with $R(G)=1$.

\begin{lem}\label{l:rad}
Let $G$ be a finite insoluble group with soluble radical $R(G)$.
\begin{itemize}\addtolength{\itemsep}{0.2\baselineskip}
\item[{\rm (i)}] $\Gamma_{\mathcal{S}}(G)$ is connected if and only if $\Gamma_{\mathcal{S}}(G/R(G))$ is connected. 
\item[{\rm (ii)}] In addition, $\delta_{\mathcal{S}}(G)= \delta_{\mathcal{S}}(G/R(G))$.
\end{itemize}
\end{lem}

\begin{proof}
Clearly, if $x,y \in G,$ then $\langle x, y\rangle$ is soluble if and only if $\langle x, y\rangle R(G)/R(G)$ is a soluble subgroup of $G/R(G)$. Consequently, $x$ and $y$ are adjacent in $\Gamma_{\mathcal{S}}(G)$ if
and only if $xR(G)$ and $yR(G)$ are adjacent in $\Gamma_{\mathcal{S}}(G/R(G)).$
\end{proof}

\begin{lem}\label{prodo}
Let $H,K$ be nontrivial finite groups with $R(H) = R(K) = 1$. 
\begin{itemize}\addtolength{\itemsep}{0.2\baselineskip}
\item[{\rm (i)}] We have $\delta_{\mathcal{S}}(H \times K) \leqs 3$.
\item[{\rm (ii)}] If $\delta_{\mathcal{S}}(H) = \delta_{\mathcal{S}}(K) = 2$, then $\delta_{\mathcal{S}}(H \times K) = 2$.
\item[{\rm (iii)}] If $\delta_{\mathcal{S}}(H) \geqs 3$ then $\delta_{\mathcal{S}}(H \wr S_2) \geqs 3$. 
\end{itemize}
\end{lem}

\begin{proof}
First consider (i). Let $x=(h_1,k_1)$ and $y = (h_2,k_2)$ be nontrivial elements of $H\times K$. Without loss of generality, we may assume $h_1\neq 1.$ If $k_2\neq 1$, then 
\[
x = (h_1,k_1) \sim (h_1,1) \sim (1,k_2) \sim (h_2,k_2) = y
\]
is a path in $\Gamma_{\mathcal{S}}(H \times K)$. Now assume $k_2 = 1$, in which case $h_2 \ne 1$. If $k_1\neq 1,$ then 
\[
x = (h_1,k_1) \sim (1,k_1) \sim (h_2,1) \sim (h_2, k_2)=y
\]
is a path in $\Gamma_{\mathcal{S}}(H \times K)$. Finally, if $k_1=k_2=1$ then let $k \in K$ be any nontrivial element and observe that  
\[
x = (h_1,1) \sim (1,k) \sim (h_2,1)=y
\]
is a path of length $2$.

Next consider part (ii). Define $x,y \in H \times K$ as above, with $h_1 \ne 1$. If $h_2 \ne 1$ then there exists $a \in H$ such that $\la h_1,a\ra$ and $\la h_2,a\ra$ are soluble and we deduce that $x \sim (a,1) \sim y$. Similarly, if $h_2=1$ then $k_2 \ne 1$ and there exists $b \in K$ such that $\la k_1, b \ra$ and $\la k_2,b\ra$ are soluble. Therefore, $x \sim (1,b) \sim y$ and we conclude that $\delta_{\mathcal{S}}(H \times K) = 2$.

Finally, let us turn to (iii). Let $x,y \in H$ be nontrivial elements with $\delta(x,y) = 3$ in $\Gamma_{\mathcal{S}}(H)$ and consider 
$a=(x,1)\s$ and $b = (y,1)\s$ in $G = H \wr S_2$, where $\s = (1,2) \in S_2$. Suppose $c = (r,s)\s^i \in G$ is adjacent to $a$ and $b$ in $\Gamma_{\mathcal{S}}(G)$, where $r,s \in H$ and $i \in \{0,1\}$. Then $c$ is
also adjacent to $a^2 = (x,x)$ and $b^2=(y,y)$. If $i=0$ then either $r \ne 1$ and $x \sim r \sim y$ in $\Gamma_{\mathcal{S}}(H)$, or $s \ne 1$ and we have $x \sim s \sim y$. This is incompatible with the condition $\delta(x,y) = 3$. Now assume $i=1$. Here $\la x, rs\ra$ and $\la y, rs\ra$ are both soluble, so $rs=1$ is the only possibility and thus $c = (r,r^{-1})\s$. However, $ac = (x,1)\s \cdot (r,r^{-1})\s=(xr^{-1},r)$ is 
adjacent to $a^2 = (x,x)$ and thus $\la x,r\ra$ is soluble. Similarly, $\la y,r \ra$ is also soluble, so $x \sim r \sim y$ in $\Gamma_{\mathcal{S}}(H)$ and once again we have reached a contradiction. We conclude that $\delta(a,b) \geqs 3$ in $\Gamma_{\mathcal{S}}(G)$ and thus $\delta_{\mathcal{S}}(G) \geqs 3$.
\end{proof}

Recall that an element in a group is \emph{real} if it is conjugate to its inverse.

\begin{lem}\label{l:easy}
Let $G$ be a nontrivial finite group with $R(G)=1$. If $x,y \in G$ are nontrivial and $|N_G(\la x \ra)|$ and $|N_G(\la y \ra)|$ are even, then $\delta(x,y) \leqs 3$. In particular, if every element in $G$ is real, then $\delta_{\mathcal{S}}(G) \leqs 3$.
\end{lem}

\begin{proof}
First observe that $x \sim z$ for all nontrivial $z \in N_G(\la x \ra)$. Therefore, if $|N_G(\la x \ra)|$ and $|N_G(\la y \ra)|$ are even, then there exist involutions $z,z' \in G$ such that $x \sim z$ and $y \sim z'$. Since $\la z, z'\ra$ is soluble, it follows that $x \sim z \sim z' \sim y$ is a path in $\Gamma_{\mathcal{S}}(G)$ and thus $\delta(x,y) \leqs 3$. Finally, note that if $x$ is real, then there exists an element $g \in G$ of even order such that $g^{-1}xg = x^{-1}$ and thus $|N_G(\la x \ra)|$ is even.  
\end{proof}

\begin{rem}\label{r:tz}
The main theorem of \cite{TZ} completely determines the finite quasisimple groups with the property that every element is real. By combining this result with Lemma \ref{l:easy}, we immediately deduce that $\delta_{\mathcal{S}}(G) \leqs 3$ when $G$ is one of the following finite simple classical groups:
\begin{itemize}\addtolength{\itemsep}{0.2\baselineskip}
\item[{\rm (a)}] ${\rm PSp}_{2m}(q)'$, where $m \geqs 1$ and $q \not\equiv 3 \imod{4}$.
\item[{\rm (b)}] $\O_{2m+1}(q)$, where $m \geqs 3$ and $q \equiv 1 \imod{4}$.
\item[{\rm (c)}] ${\rm P\O}_{4m}^{\e}(q)$, where $m \geqs 2$ (and $q \not\equiv 3 \imod{4}$ if $\e=+$ and $m \geqs 3$). 
\end{itemize}
\end{rem}

Let $G$ be a finite group and let $H$ be a core-free subgroup. Recall that the \emph{base size} for the natural action of $G$ on the set of cosets $G/H$, denoted $b(G,H)$ is defined to be the minimal number $b$ of conjugates $H^{g_1}, \ldots, H^{g_b}$ such that $\bigcap_iH^{g_i} = 1$. In particular, $b(G,H)=2$ if and only if $H \ne 1$ and $H \cap H^g=1$ for some $g \in G$. 

\begin{lem}\label{l:base}
Let $G$ be a nontrivial finite group with $R(G)=1$ and suppose there exists a nontrivial element $x \in G$ that is contained in a unique maximal subgroup $H$ of $G$. If $H$ is core-free and $b(G,H) = 2$, then $\delta_{\mathcal{S}}(G) \geqs 3$.
\end{lem}

\begin{proof}
First observe that if $y \in G \setminus H$ then $\la x,y \ra$ is equal to $G$, which is insoluble, and thus $B_1(x) \subseteq H$. Since $b(G,H) = 2$, there exists $g \in G$ such that $H \cap H^g = 1$, which implies that $B_1(x) \cap B_1(x^g)$ is empty. Therefore, $\delta(x,x^g) \geqs 3$ and the result follows.  
\end{proof}

\begin{ex}\label{e:mon}
In order to demonstrate the utility of the previous lemma, let $G = \mathbb{M}$ be the Monster sporadic simple group and let $x \in G$ be an element of order $59$. As explained in the proof of \cite[Theorem 4.1]{BH}, $x$ is contained in a unique maximal subgroup $H = {\rm L}_{2}(59)$ of $G$ and by applying the main theorem of \cite{BOW} we deduce that $b(G,H) = 2$. Therefore, Lemma \ref{l:base} implies that 
$\delta_{\mathcal{S}}(G) \geqs 3$.
\end{ex}

\subsection{Computational methods}\label{ss:comp}

In this section we discuss the computational methods that play a key role in the proofs of our main results. All of our computations are performed using {\sc Magma} (version 2.26-6) \cite{magma} and we are mainly interested in the case where $G$ is almost simple. So for the remainder of this section, let us assume $G$ is an almost simple group that is amenable to direct computations in {\sc Magma}.

Let $x \in G$ be nontrivial and suppose we seek to determine $B_{\ell}(x)$, the ball of radius $\ell$ in $\Gamma_{\mathcal{S}}(G)$, centred at $x$. To do this, we first use the command \texttt{AutomorphismGroupSimpleGroup} to construct $G$ as a permutation group and we then identify $x$ (up to conjugacy) as an element of $G$ in this representation (for example, via the \texttt{ConjugacyClasses} command, or by random search). Next we use the function \texttt{SolubleSubgroups} to determine a set of representatives of the conjugacy classes of soluble subgroups of $G$ with order divisible by $|x|$. Given such a subgroup $H$, we then compute the number $n$ of conjugates of $H$ containing $x$, noting that 
\[
n = \frac{|x^G \cap H|}{|x^G|}\cdot |G: N_G(H)|.
\]
From here, we can then construct all the soluble overgroups of $x$ and by taking the union of these subgroups, excluding the identity element, we return $B_1(x)$. 

Similarly, in order to construct $B_2(x)$ we first determine a set of representatives $x_1, \ldots, x_k$ of the distinct conjugacy classes in $G$ that meet $B_1(x)$. Set $y = x_1$. As above, we construct $B_1(y)$ and then for each conjugate $y^g \in B_1(x)$ we obtain $B_1(y^g) = B_1(y)^g$ and we take the union of these sets. We can now construct $B_2(x)$ by repeating this process for $x_2, \ldots, x_k$ and this approach can be extended to give $B_{\ell}(x)$ for any $\ell$. 

\begin{ex}\label{e:m23}
Let $G$ be the Mathieu group ${\rm M}_{23}$ and let $G^{\#}$ be the set of nontrivial elements in $G$. We claim that $\delta_{\mathcal{S}}(G)=4$.

Here the function \texttt{AutomorphismGroupSimpleGroup} returns $G$ in its natural permutation representation of degree $23$. First we use \texttt{ConjugacyClasses} to determine a set of representatives of the conjugacy classes of $G$. Then for each representative $x$, we compute $|N_G(\la x \ra)|$, which is even unless $|x| \in \{11,23\}$. If we fix an element $x$ of order $23$, then $B_1(x)$ coincides with the nontrivial elements in $N_G(\la x \ra) = 23{:}11$ and by random search we can find a conjugate $y = x^g$ such that $\la a, b \ra$ is insoluble for all $a \in B_1(x)$, $b \in B_1(y)$. This implies that $\delta(x,y) \geqs 4$ and thus $\delta_{\mathcal{S}}(G) \geqs 4$. Therefore, in order to conclude that $\delta_{\mathcal{S}}(G)=4$ it suffices to show that $B_4(x) = G^{\#}$ when $|x| \in \{11,23\}$ (see Lemma \ref{l:easy}). 

First assume $|x|=11$. By applying the procedure described above, we compute 
\[
|B_1(x)| = 1264, \; |B_2(x)| = 135629, \; |B_3(x)| = 9540519.
\]
We now  implement the following exhaustive process in order to show that $B_4(x) = G^{\#}$. First we express each element $y \in B_3(x)$ in the form $x_i^{g}$, where $x_1, \ldots, x_k$ represent the distinct conjugacy classes in $G$ that meet $B_3(x)$. This allows us to express $B_3(x)$ as a disjoint union 
\[
B_3(x) = \bigcup_{i=1}^k \{ x_i^{g} \, :\, g \in T_i \},
\]
where the elements $x_i$ and the sets $T_i$ are explicitly determined.
Then starting with $Y = G^{\#}$, we choose $y \in Y$ at random and then we use random search to find $i \in \{1, \ldots, k\}$ and $g \in T_i$ such that $\la x_i^g, y\ra$ is soluble. It follows that $\la x_i^h , y^{g^{-1}h} \ra$ is soluble for all $h \in T_i$ and so we redefine $Y$ by removing all elements of the form $y^{g^{-1}h}$ with $h \in T_i$. We now repeat the process, which eventually terminates when $Y$ reaches the empty set. This allows us to conclude that each $y \in G^{\#}$ is adjacent to an element in $B_3(x)$ and thus $B_4(x) = G^{\#}$ as claimed. An entirely similar argument applies when $|x|=23$, noting that 
\[
|B_1(x)| = 252, \; |B_2(x)| = 23528, \; |B_3(x)| = 1858031.
\]
\end{ex}

\begin{ex}\label{e:m24}
There are variations of the computational approach outlined in Example \ref{e:m23}, which can be more efficient for certain groups. Specifically, the approach presented below does not require the construction of any balls of radius $3$, which can be expensive in terms of time and memory. 

For example, suppose we seek to show that $\delta_{\mathcal{S}}(G) = 4$ when $G$ is the largest Mathieu group ${\rm M}_{24}$. Here we work with the natural permutation representation of degree $24$ and we first check that if $x \in G$ and $|x| \ne 23$ then $|N_G(\la x \ra)|$ is even. Fix $x \in G$ of order $23$. By random search we can find an element $g \in G$ such that $\la a,b \ra$ is insoluble for all $a \in B_1(x)$, $b \in B_1(x^g)$ and thus $\delta(x,x^g) \geqs 4$. Therefore, in order to conclude that $\delta_{\mathcal{S}}(G) = 4$ it suffices to show that any two elements of order $23$ are connected by a path of length at most $4$. 

To do this, we can proceed as follows. Let $Y$ be the set of elements of order $23$ in $G$, so $|Y| = 21288960$ and $Y$ is a union of two conjugacy classes, labelled \texttt{23A} and \texttt{23B}. Fix an element $x \in \texttt{23A}$ and construct $B_2(x)$ as above (we find that $|B_1(x)| = 252$ and $|B_2(x)| = 50093$). We now initiate the following process. Choose $y \in Y$ at random and construct $B_1(y)$. By random search, find 
$a \in B_2(x)$ and $b \in B_1(y)$ such that $\la a,b \ra$ is soluble. It follows that $\delta(x,z) \leqs 4$ for all 
\[
z \in \{(y^g)^m \, :\, g \in C_G(x), \, 1 \leqs m \leqs 22\}.
\]
We now remove this subset from $Y$ and we continue to repeat this process until $Y$ is empty. This allows us to conclude that $B_4(x) = G^{\#}$ for all $x \in \texttt{23A}$. Finally,  since the Sylow $23$-subgroups of $G$ are cyclic, we immediately deduce that the same conclusion holds if $x \in \texttt{23B}$.
\end{ex}

\section{A reduction theorem}\label{s:red}

In this section, we reduce the proofs of Theorems \ref{t:main2} and \ref{t:main3} to almost simple groups.

\begin{lem}\label{mono}
Let $G$ be a finite group with $R(G)=1$ and socle $N_1\times \cdots \times N_k$, where each $N_i$ is a non-abelian minimal normal subgroup of $G$. If $k \geqs 2$ then $\Gamma_{\mathcal{S}}(G)$ is connected and $\delta_{\mathcal{S}}(G) \leqs 3$.
\end{lem}

\begin{proof}
Let $x,y \in G$ be distinct nontrivial elements. Then there exist nontrivial elements $s \in N_1$, $t \in N_2$ such that $[x,s]=[y,t]=1$ (see \cite[Theorem 1.48]{gor}, for example) and we deduce that $x \sim s \sim t \sim y$ is a path in $\Gamma_{\mathcal{S}}(G)$.
\end{proof}

\begin{lem}\label{l:d2}
Let $G = T^k$, where $T$ is a non-abelian simple group and $k \geqs 2$. Let 
\[
x = (x_1, \ldots, x_k), \;\; y = (y_1, \ldots, y_k)
\]
be distinct nontrivial elements of $G$ with $x_1 = 1$. Then $\delta(x,y) \leqs 2$.
\end{lem}

\begin{proof}
If $y_1\neq 1,$ then $x \sim (y_1,1,\ldots,1) \sim y$ is a path in $\Gamma_{\mathcal{S}}(G)$. Similarly, if $y_1=1$ then we may consider the path $x \sim (t,1,\ldots,1)\sim y$, where $t \in T$ is an arbitrary nontrivial element. 
\end{proof} 

We are now ready to establish the main result of this section.

\begin{thm}\label{t:nas}
Let $G$ be a finite insoluble group. If $G$ is not almost simple, then $\Gamma_{\mathcal{S}}(G)$ is connected and $\delta_{\mathcal{S}}(G) \leqs 3$.
\end{thm}

\begin{proof}
In view of Lemmas \ref{l:rad} and \ref{mono}, we may assume $G$ is monolithic with socle $N = T^k$, where $T$ is a non-abelian simple group and $k \geqs 2$. Let us identify $G$ with a subgroup of ${\rm Aut}(N) = {\rm Aut}(T) \wr S_k$ and let $x,y \in G$ be nontrivial elements, say $x = (x_1, \ldots, x_k)\sigma$ with $x_i \in {\rm Aut}(T)$ and $\sigma \in S_k$. Write $\sigma= \sigma_1\cdots \sigma_d$ as a product of disjoint cycles (including cycles of length $1$). We may assume $\sigma_1=(1,2,\ldots,u)$ with $u \geqs 1$. Set $a = x_1 \cdots x_u \in {\rm Aut}(T)$ and define $z = (t, 1, \ldots, 1) \in N$, where $1 \ne t \in C_T(a)$ (the existence of $t$ follows from \cite[Theorem 1.48]{gor}). Next observe that 
\[
Y = \la z, z^{x}, \ldots, z^{x^{u-1}} \ra \leqs \la z \ra \times \la z^x \ra \times \cdots \times \la z^{x^{u-1}} \ra
\]
is an abelian group normalised by $\la x \ra$, so $\la Y, x \ra$ is a soluble subgroup of $G$ containing $x$ and $z$, whence $\la x,z \ra$ is soluble. 

Now, if $y \in N$ then $\delta(z,y) \leqs 2$ by Lemma \ref{l:d2} and therefore $\delta(x,y) \leqs 3$. If $y \notin N,$ then as above we may construct a nontrivial element $z' \in N$ of the form $z' = (1,t',1,\ldots,1)$ for a suitable $t' \in T$ with $z' \sim y$. This gives a path $x \sim z \sim z' \sim y$ and the proof is complete. 
\end{proof}

\begin{cor}\label{c:nas}
Fix $\ell \in \{2,3\}$. Then there are infinitely many insoluble groups $G$ that are not almost simple with $\delta_{\mathcal{S}}(G) = \ell$.
\end{cor}

\begin{proof}
This follows by combining Theorem \ref{t:nas} with Lemma \ref{prodo} and part (ii) of Theorem \ref{t:main2}. First observe that the latter result shows that there are infinitely many finite simple groups $T$ with $\delta_{\mathcal{S}}(T) = \ell$. If $\delta_{\mathcal{S}}(T) =2$ then Lemma \ref{prodo}(ii) implies that $\delta_{\mathcal{S}}(T \times T) = 2$. On the other hand, if $\delta_{\mathcal{S}}(T) =3$ then Lemma \ref{prodo}(iii) yields $\delta_{\mathcal{S}}(T \wr S_2) \geqs 3$, which means that $\delta_{\mathcal{S}}(T \wr S_2)=3$ by Theorem \ref{t:nas}. 
\end{proof}

\begin{rem}\label{r:mon}
Let $G$ be a finite monolithic group with socle $T^k$ for some non-abelian simple group $T$ and positive integer $k \geqs 2$. Here we record the fact that there are examples with 
$\delta_{\mathcal{S}}(G) = \delta_{\mathcal{S}}(T) \pm 1$.
\begin{itemize}\addtolength{\itemsep}{0.2\baselineskip}
\item[{\rm (a)}] Theorem \ref{t:main2} demonstrates the existence of simple groups $T$ with $\delta_{\mathcal{S}}(T) = 4$, while the bounds in Lemma \ref{prodo}(iii) and Theorem \ref{t:nas} show that 
$\delta_{\mathcal{S}}(T \wr S_2) = 3$.
\item[{\rm (b)}] Suppose $T = {\rm U}_4(2)$, $H = {\rm U}_{4}(2).2$ and $G = H \wr S_2$. Here one can check that $\delta_{\mathcal{S}}(T) = 2$ and $\delta_{\mathcal{S}}(H) = 3$, whence $\delta_{\mathcal{S}}(G) = 3$ by combining Lemma \ref{prodo}(iii) with Theorem \ref{t:nas}.
\end{itemize}
\end{rem}

We can also reduce the proof of Theorem \ref{t:main3} to almost simple groups.

\begin{prop}\label{p:nas}
Let $G$ be a finite insoluble group with $R(G)=1$ that is not almost simple and let $x \in G$ be nontrivial. Then there exists an involution $y \in G$ with $\delta(x,y) \leqs 2$. 
\end{prop}

\begin{proof}
Let $N = T^k$ be a minimal normal subgroup of $G$, where $T$ is a finite non-abelian simple group and $k$ is a positive integer. It follows from the proof of Theorem \ref{t:nas} that there exists $1\neq t \in T$ such that $x$ is adjacent to $n=(t,1,\ldots, 1)\in N.$ If $k\geq 2,$ then $n$ is adjacent to $(1,u,\ldots,u)\in N$, where $u \in T$ is an arbitrary involution. And if $k=1$ then $G$ contains another minimal normal subgroup $M \neq N$ and $n$ is adjacent to every involution in $M$.
\end{proof}

We close this section with the following result, which we will use to show that the Mathieu group ${\rm M}_{12}$ is the smallest finite group $G$ with $\delta_{\mathcal{S}}(G)>3$ (see Corollary \ref{c:m12}).

\begin{prop}\label{p:small}
Let $G$ be a finite insoluble group with $|G|<|{\rm M}_{12}|$. Then $\delta_{\mathcal{S}}(G) \leqs 3$.
\end{prop}

\begin{proof}
In view of Theorem \ref{t:nas}, we may assume $G$ is almost simple with socle $G_0$. By inspecting the list of non-abelian simple groups of order less than $|{\rm M}_{12}| = 95040$, we deduce that 
\[
G_0 \in \{ A_7, {\rm L}_{2}(q), {\rm L}_{3}^{\e}(3), {\rm L}_{3}^{\e}(4), {\rm L}_{4}^{\e}(2), {}^2B_2(8), {\rm M}_{11} \}
\]
up to isomorphism, where $5 \leqs q \leqs 53$. For the groups with socle ${\rm L}_{2}(q)$, we refer the reader to Theorem \ref{t:psl2} where the precise diameter of $\Gamma_{\mathcal{S}}(G)$ is determined. In each of the remaining cases, we can use the computational approach outlined in Section \ref{ss:comp} to calculate the precise diameter of $\Gamma_{\mathcal{S}}(G)$. In this way, we obtain 
\[
\delta_{\mathcal{S}}(G)  = \left\{ \begin{array}{ll}
2 & \mbox{if $G = {\rm L}_{3}(3).2$ or ${\rm U}_{4}(2)$} \\
3 & \mbox{otherwise.}
\end{array}\right.
\]
The result follows.
\end{proof}

\section{Sporadic groups}\label{s:spor}

In this section we prove Theorems \ref{t:main2} and \ref{t:main3} in the case where $G$ is an almost simple group with socle a sporadic group. Our main theorem for simple sporadic groups is as follows.

\begin{thm}\label{t:spor}
If $G$ is a simple sporadic group, then $3 \leqs \delta_{\mathcal{S}}(G) \leqs 5$. In particular, the following hold: 
\begin{itemize}\addtolength{\itemsep}{0.2\baselineskip}
\item[{\rm (i)}] If $G \in \{ {\rm M}_{11}, {\rm J}_1, {\rm J}_2, {\rm Suz},  {\rm He}, {\rm Ru}, {\rm Fi}_{22}, {\rm Ly}, {\rm J}_4 \}$ then $\delta_{\mathcal{S}}(G) = 3$.
\item[{\rm (ii)}] If $G \in \{ {\rm M}_{12}, {\rm M}_{22}, {\rm M}_{23}, {\rm M}_{24}, {\rm HS}, {\rm J}_3 \}$ then $\delta_{\mathcal{S}}(G) = 4$. 
\item[{\rm (iii)}] If $G \in \{ {\rm Co}_2, {\rm Co}_3, {\rm McL}, \mathbb{B}\}$ then $\delta_{\mathcal{S}}(G) \in \{4,5\}$.
\end{itemize}
\end{thm}

We will prove Theorem \ref{t:spor} in a sequence of lemmas. The arguments rely heavily on the computational methods (using {\sc Magma} \cite{magma}) discussed in Section \ref{ss:comp}.

\begin{lem}\label{l:spor1}
We have
\[
\delta_{\mathcal{S}}(G) = \left\{ 
\begin{array}{ll}
3 & \mbox{if $G \in \{ {\rm M}_{11}, {\rm J}_1, {\rm J}_2, {\rm Suz},  {\rm He}, {\rm Ru}, {\rm Fi}_{22} \}$}\\
4 & \mbox{if $G \in \{ {\rm M}_{12}, {\rm M}_{22}, {\rm M}_{23}, {\rm M}_{24}\}$.}
\end{array}\right.
\]
\end{lem}

\begin{proof}
If $G$ is a Mathieu group, then we can use {\sc Magma} to compute the precise diameter of $\Gamma_{\mathcal{S}}(G)$ (see Section \ref{ss:comp}). In the remaining cases, it is easy to check that $|N_G(\la x \ra)|$ is even for all $x \in G$ and thus $\delta_{\mathcal{S}}(G) \leqs 3$ by Lemma \ref{l:easy}. Moreover, we can use random search to find two conjugate elements $x,y \in G$ of order $r$ such that $B_1(x) \cap B_1(y)$ is empty, where $r$ is the largest prime divisor of $|G|$. Therefore $\delta(x,y) \geqs 3$ and we deduce that $\delta_{\mathcal{S}}(G) = 3$ as required.
\end{proof}

By combining Lemma \ref{l:spor1} with Proposition \ref{p:small}, we obtain the following corollary.

\begin{cor}\label{c:m12}
The Mathieu group ${\rm M}_{12}$ is the smallest finite group $G$ with $\delta_{\mathcal{S}}(G)>3$.
\end{cor}

\begin{lem}\label{l:spor2}
We have $\delta_{\mathcal{S}}(G) \geqs 4$ if $G \in \{ {\rm HS}, {\rm J}_3, {\rm Co}_{2}, {\rm Co}_3, {\rm McL}, \mathbb{B}\}$.
\end{lem}

\begin{proof}
For now let us assume $G \ne \mathbb{B}$. In each of these cases we can proceed as in Examples \ref{e:m23} and \ref{e:m24}. For example, suppose $G = {\rm HS}$ and let $x \in G$ be an element of order $11$. Then $B_1(x)$ coincides with the set of nontrivial elements in $N_G(\la x \ra) = 11{:}5$ and by random search we can find $g \in G$ such that $B_1(x) \cap B_1(y)$ is empty, where $y=x^g$. Moreover, we can find such an element $y$ with the property that $\la a, b\ra$ is insoluble for all $a \in B_1(x)$, $b \in B_1(y)$. This means that $\delta(x,y) \geqs 4$ and thus $\delta_{\mathcal{S}}(G) \geqs 4$. A very similar argument applies for the remaining groups, working with elements of order $19$, $23$,  $23$ and $11$ for $G = {\rm J}_3$, ${\rm Co}_2$, ${\rm Co}_3$ and ${\rm McL}$, respectively.

Finally, let us assume $G = \mathbb{B}$ is the Baby Monster and recall the definitions of the intersection graph and non-generating graph of $G$ (see Remark \ref{r:ng} in Section \ref{s:intro}). In \cite{Free}, Freedman proves that the intersection graph of $G$ has diameter $5$, noting that the shortest path between two specific subgroups of order $47$ has length $5$. This implies that the soluble intersection graph of $G$ has diameter at least $5$ (see the paragraph preceding Corollary \ref{c:int}) and thus $\delta_{\mathcal{S}}(G) \geqs 4$ by \cite[Proposition 12.2]{Cam} since the soluble intersection graph and the soluble graph form a dual pair in the sense of \cite[Section 12.1]{Cam}.
\end{proof}

\begin{lem}\label{l:spor3}
We have $\delta_{\mathcal{S}}(G) = 4$ if $G = {\rm HS}$ or ${\rm J}_3$.
\end{lem}

\begin{proof}
In view of Lemma \ref{l:spor2}, it suffices to show that $\delta_{\mathcal{S}}(G) \leqs 4$ and in both cases we follow the computational approach presented in Example \ref{e:m24}. 

First assume $G = {\rm HS}$ and let $x \in G$ be nontrivial, noting that $|N_G(\la x \ra)|$ is even unless $|x|=11$. Let $Y$ be the set of elements of order $11$ in $G$, so $|Y| =  8064000$ and we fix an element $x \in Y$. It is sufficient to show that $\delta(x,y) \leqs 4$ for all $y \in Y$. First we construct $B_2(x)$ and $B_1(y)$ for some randomly chosen $y \in Y$. Then by random search we find $a \in B_2(x)$ and $b \in B_1(y)$ such that $\la a,b \ra$ is soluble and thus $x$ is connected by a path of length at most $4$ to every element in the set
\[
\{ (y^g)^m \,:\, g \in C_G(x), \, 1 \leqs m \leqs 10\}.
\]
We now redefine $Y$ by removing the elements in this set and we repeat the process. Eventually, we reduce $Y$ to the empty set and the result follows.

The case $G = {\rm J}_3$ can be handled in an entirely similar fashion, working with the set of elements of order $19$ in $G$.
\end{proof}

\begin{lem}\label{l:spor4}
We have $\delta_{\mathcal{S}}(G) \geqs 3$ if $G \in \{ {\rm Co}_1, {\rm HN}, {\rm O'N}, {\rm Fi}_{23}, {\rm Fi}_{24}', {\rm Th}, {\rm Ly}, {\rm J}_4,  \mathbb{M}\}$.
\end{lem}

\begin{proof}
First assume $G = {\rm Co}_1$ and fix $x \in G$ of order $26$. Then $H = (A_4 \times G_2(4)){:}2$ is the unique maximal subgroup of $G$ containing $x$ (see \cite[Table 1]{BH}) and thus $B_1(x)$ is contained in $H$. With the aid of {\sc Magma}, this observation allows us to construct $B_1(x)$ by working inside $H$ (we get $|B_1(x)| = 1871$) and then by random search we can find an element $g \in G$ such that $B_1(x) \cap B_1(x^g)$ is empty. This implies that $\delta(x,x^g) \geqs 3$ and the result follows. A very similar argument applies when $G = {\rm HN}$. Here we take $x \in G$ of order $22$, noting that $x$ is contained in a unique maximal subgroup $H = 2.{\rm HS}.2$ of $G$ and thus $B_1(x) \subseteq H$. By working in $H$, we can determine $B_1(x)$ and then conclude as in the previous case.

Next assume $G = {\rm O'N}$ and $x \in G$ has order $31$. Now $x$ is contained in precisely two maximal subgroups of $G$, which are non-conjugate copies of ${\rm L}_{2}(31)$. So if $H = {\rm L}_2(31)$ is a maximal subgroup of $G$ containing $x$, then $N_H(\la x \ra) = 31{:}15$ is a maximal subgroup of $H$ and we deduce that $B_1(x) \subseteq N_G(\la x \ra) = N_H(\la x \ra)$. By inspecting \cite{BOW} we see that the base size $b(G,H)$ for the action of $G$ on $G/H$ is $2$. This means that there exists an element $g \in G$ such that $H \cap H^g=1$ and thus $\delta(x,x^g) \geqs 3$.

In each of the remaining cases, we can proceed as in the proof of \cite[Theorem 4.1]{BH} to find an element in $G$ of order $r$ that is contained in a unique maximal subgroup $H$, where $r$ and $H$ are as follows: 
\[
\begin{array}{ccccccc} \hline
G & {\rm Fi}_{23} & {\rm Fi}_{24}' & {\rm Th} & {\rm Ly} & {\rm J}_4 & \mathbb{M} \\ \hline
r & 35 & 29 & 39 & 67 & 43 & 59 \\ 
H & S_{12} & 29{:}14 & (3 \times G_2(3)){:}2 & 67{:}22 & 43{:}14 & {\rm L}_{2}(59) \\ \hline
\end{array}
\]
Then by inspecting \cite{BOW} we see that $b(G,H) = 2$ in each case and this allows us to conclude via Lemma \ref{l:base}. 
\end{proof}

\begin{lem}\label{l:spor5}
We have $\delta_{\mathcal{S}}(G) = 3$ if $G = {\rm Ly}$ or ${\rm J}_4$.
\end{lem}

\begin{proof}
By the previous lemma, it suffices to show that $\delta_{\mathcal{S}}(G) \leqs 3$. Therefore, in view of Lemma \ref{l:easy}, it is sufficient to prove that $|N_G(\la x \ra)|$ is even for all nontrivial $x \in G$ and this is how we proceed.

First assume $G = {\rm Ly}$ and let $x \in G$ be nontrivial. Here $|C_G(x)|$ is even unless 
\[
|x| \in \{15,25,31,33,37,67\}
\]
and so we may assume $|x|$ is one of these possibilities. If $|x|=37$ then $N_G(\la x \ra) = 37{:}18$ is a maximal subgroup of $G$. Similarly, $N_G(\la x \ra) = 67{:}22$ is a maximal subgroup when $|x|=67$. If $|x| \in \{15,25,31\}$ then we may embed $x$ in a maximal subgroup $H = G_2(5)$ and with the aid of {\sc Magma} it is straightforward to check that $N_H(\la x \ra)$ has even order. Finally, if $|x|=33$ then we embed $x$ in a maximal subgroup $H = 3.{\rm McL}.2$ and we find that $|N_H(\la x \ra)| = 330$ is even. 

Now suppose $G = {\rm J}_4$. Here we first observe that $|C_G(x)|$ is even unless 
\[
|x| \in \{23,29,31,35, 37,43\}.
\]
If $|x| \in \{29,37,43\}$ then $H = N_G(\la x \ra)$ is a maximal subgroup of $G$ with even order. If $|x|=23$ then $x$ is contained in a maximal subgroup $H = {\rm L}_{2}(23).2$ and once again we deduce that $|N_G(\la x\ra)|$ is even. Similarly, if $|x|=31$ then $x \in H = {\rm L}_{2}(32).5$ and $N_H(\la x \ra) = 31{:}10$. Finally, let us assume $|x|=35$. Here $x$ is contained in a maximal subgroup $H = 2^{3+12}.(S_5 \times {\rm L}_{3}(2))$ and using {\sc Magma} we deduce that $|N_H(\la x \ra)| = 420$ is even. 
\end{proof}

To complete the proof of Theorem \ref{t:spor}, and also the proof of Theorem \ref{t:main2} for sporadic groups, it remains to show that $\delta_{\mathcal{S}}(G) \leqs 5$ for every almost simple sporadic group. Since we handled the group ${\rm M}_{23}$ in Lemma \ref{l:spor1}, this is an immediate consequence of the following result, which also establishes Theorem \ref{t:main3} in this setting. 

\begin{prop}\label{p:spor}
Let $G$ be an almost simple sporadic group and let $x \in G$ be nontrivial. Then either
\begin{itemize}\addtolength{\itemsep}{0.2\baselineskip}
\item[{\rm (i)}] There exists an involution $y \in G$ such that $\delta(x,y) \leqs 2$; or
\item[{\rm (ii)}] $G = {\rm M}_{23}$ and $|x|=23$.
\end{itemize}
\end{prop}

\begin{proof}
Let $G_0$ be the socle of $G$ and observe that the conclusion in part (i) holds if for all nontrivial $x \in G$, there exists $z \in N_G(\la x \ra)$ with $|N_G(\la z \ra)|$ even. This property is very straightforward to verify using {\sc Magma} in the following cases:
\[
G_0 \in \{ {\rm M}_{11},{\rm M}_{12},{\rm M}_{22},{\rm M}_{24}, {\rm J}_1, {\rm J}_2, {\rm J}_3, {\rm He}, {\rm McL}, {\rm Suz}, {\rm Ru}, {\rm HS}, {\rm Co}_2, {\rm Co}_3, {\rm Fi}_{22}, {\rm Fi}_{23} \}.
\]
The same property also holds if $G = {\rm M}_{23}$ and $|x| \ne 23$ and we refer the reader to Remark \ref{r:m23} in Section \ref{s:intro} for further comments on the special case recorded in part (ii), which is a genuine exception. The desired result for $G = {\rm Ly}$ and ${\rm J}_4$ follows immediately from the proof of Lemma \ref{l:spor5}.

Next assume $G = {\rm Th}$ and let $x \in G$ be nontrivial. By inspecting the Web Atlas \cite{WebAt}, we see that $|C_G(x)|$ is odd only if $|x| \in \{9,13,19,21,27,31,39\}$. In particular, if $x^m$ has order $3$, then $|C_G(x^m)|$ is even and there is a path $x \sim x^m \sim y$ with $y \in C_G(x^m)$ an involution. Therefore, we may assume $|x| \in \{13,19,31\}$. By inspecting the list of maximal subgroups of $G$ (see \cite{WebAt}), we deduce that $x$ is contained in a maximal subgroup $H$ such that $|N_H(\la x \ra)|$ is divisible by $3$ and so we can complete the argument as before. For example, if $|x|=13$ then $x \in H = {\rm L}_{3}(3)$ and $N_H(\la x \ra) = 13{:}3$. 

The remaining groups can be handled in a very similar fashion and we omit the details. 
\end{proof}

We present the following result to conclude our analysis of sporadic groups, which can be viewed as an extension of Theorem \ref{t:main2}(ii). 

\begin{thm}\label{t:spor2}
Let $G$ be an almost simple sporadic group with socle $G_0$. If $G \ne G_0$ then $\delta_{\mathcal{S}}(G) = 3$.
\end{thm}

\begin{proof}
First assume $G$ is one of the following groups:
\[
{\rm M}_{12}.2, \, {\rm M}_{22}.2, \, {\rm HS}.2, \, {\rm J}_2.2, \, {\rm McL}.2, \, {\rm Suz}.2, \, {\rm He}.2, \, {\rm J_3}.2.
\]
In each case, it is straightforward to check that $|N_G(\la x \ra)|$ is even for all nontrivial $x \in G$ and thus $\delta_{\mathcal{S}}(G) \leqs 3$ by Lemma \ref{l:easy}. Moreover, if $x \in G$ has order $r$, where $r$ is the largest prime divisor of $|G|$, then $B_1(x)$ coincides with the set of nontrivial elements in $H = N_G(\la x \ra)$ and by random search we can find $g \in G$ such that $H \cap H^g = 1$. This implies that $B_1(x) \cap B_1(x^g)$ is empty and we conclude that $\delta_{\mathcal{S}}(G) = 3$. The same argument also applies when $G = {\rm Fi}_{22}.2$, working with an element $x \in G$ of order $21$.  

Next assume $G = {\rm O'N}.2$. By inspection we see that $|C_G(x)|$ is odd if and only if $|x|=31$. But if $x$ has order $31$ then $N_G(\la x \ra) = 31{:}30$ and so in view of Lemma \ref{l:easy} we deduce that $\delta_{\mathcal{S}}(G) \leqs 3$. In order to establish equality, let $x \in G$ be an element of order $31$ and observe that $H = N_G(\la x \ra)$ and $G_0$ are the only maximal subgroups of $G$ containing $x$. Suppose $y \in B_1(x)$. If $y \not\in G_0$ then the previous observation implies that $y \in H$. On the other hand, if $y \in G_0$ then $y \in B_1(x) \cap G_0 = N_{G_0}(\la x \ra)$, as noted in the proof of Lemma \ref{l:spor4}. Therefore, $B_1(x)$ is the set of nontrivial elements in $H$. By the main theorem of \cite{BOW}, we know that $b(G,H) = 2$ and we conclude by applying Lemma \ref{l:base}. 

Now suppose $G = {\rm Fi}_{24}$. Here one can check that $|C_G(x)|$ is even unless $|x| \in \{27,29,39,45\}$. But in each case, we can construct $N_G(\la x \ra)$ and we find that this normaliser has even order. This establishes the bound $\delta_{\mathcal{S}}(G) \leqs 3$. Now, if $x \in G$ has order $29$ then $H = N_G(\la x \ra) = 29{:}28$ and $G_0$ are the only maximal subgroups of $G$ containing $x$ and as in the previous case we deduce that $B_1(x)$ coincides with the nontrivial elements in $H$ (note that $N_{G_0}(\la x \ra) = 29{:}14$ is the unique maximal overgroup of $x$ in $G_0$, as recorded in the proof of Lemma \ref{l:spor4}). We now complete the argument as before, noting that $b(G,H) = 2$ by \cite{BOW}. 

Finally, let us assume $G = {\rm HN}.2$ and note that $|C_G(x)|$ is even unless $|x| \in \{19,25,35\}$. If $|x|=35$ then we calculate that $|N_G(\la x \ra)| = 840$. Similarly, if $|x|=19$ then we may embed $x$ in a maximal subgroup $H = {\rm U}_3(8){:}6$ and it is easy to check that $|N_H(\la x \ra)| = 342$. Finally, suppose $|x|=25$. There is a unique conjugacy class of such elements in $G$ and we may embed $x$ in a maximal subgroup $H = 5^{2+1+2}.4.A_5.2$. We can work with generators given in the Web Atlas \cite{WebAt} in order to construct $H$ and one can then check that $|N_H(\la x \ra)| = 500$. We conclude that $\delta_{\mathcal{S}}(G) \leqs 3$. To show that $\delta_{\mathcal{S}}(G)=3$, let $x \in G$ be an element of order $22$ and note that $x$ is contained in exactly two maximal subgroups of $G$, namely $H = 4.{\rm HS}.2$ and $G_0$. As noted in \cite[Table 1]{BH}, $H_0 = H \cap G_0 = 2.{\rm HS}.2$ is the unique maximal overgroup of $x$ in $G_0$ and we deduce that $B_1(x) \subseteq H$. We now construct $H$ (using the generators in \cite{WebAt}) and we determine the subset $B_1(x)$ by working inside $H$. We get $|B_1(x)| = 439$ and by random search we can find an element $g \in G$ such that $B_1(x) \cap B_1(x^g)$ is empty. This implies that  $\delta(x, x^g) \geqs 3$ and the result follows.
\end{proof} 

\section{Symmetric and alternating groups}\label{s:alt}

In this section we establish Theorems \ref{t:main2} and \ref{t:main3} when $G$ is almost simple with socle $G_0 = A_n$. Let us first observe that if $n \in \{5,6\}$ then it is straightforward to check that 
\begin{equation}\label{e:a56}
\delta_{\mathcal{S}}(G) = \left\{\begin{array}{ll}
3 & \mbox{if $G = A_6$, $S_6$ or ${\rm M}_{10}$} \\
2 & \mbox{otherwise}
\end{array}\right.
\end{equation}
and the conclusion in part (i) of Theorem \ref{t:main3} holds. So for the remainder we may assume $n \geqs 7$.

\begin{thm}\label{t:sym}
Let $G = S_n$ with $n \geqs 7$. Then $\delta_{\mathcal{S}}(G) = 3$ and every vertex in $\Gamma_{\mathcal{S}}(G)$ is adjacent to an involution.
\end{thm}

\begin{proof}
First observe that every element in $G$ is real (that is, $x$ and $x^{-1}$ are conjugate for all $x \in G$), whence $|N_G(\la x\ra)|$ is even for all $x \in G$ and thus $\delta_{\mathcal{S}}(G) \leqs 3$ by Lemma \ref{l:easy}. In particular, every vertex in $\Gamma_{\mathcal{S}}(G)$ is adjacent to an involution. To complete the proof, we need to show that $\delta_{\mathcal{S}}(G) \geqs 3$. 

First assume $n=p$ is a prime and let $x \in G$ be a $p$-cycle. If $H$ is a soluble subgroup of $G$ containing $x$, then \cite[Theorem 1.2]{Jones} implies that $H \leqs N_G(\la x \ra) = L$, where $L = {\rm AGL}_{1}(p)$. Therefore, $B_1(x)$ coincides with the set of nontrivial elements in $L$ and we note that $b(G,L) = 2$ by \cite{BGS}. Therefore, Lemma \ref{l:base} implies that $\delta_{\mathcal{S}}(G) \geqs 3$, as required. 

Finally, suppose $n \geqs 8$ is composite. The cases $n \in \{8,9\}$ can be handled directly with {\sc Magma} (also see the proof of Proposition \ref{p:small}). Now assume $n \geqs 10$ and fix a prime $p$ such that $n/2< p < n-2$ (the existence of such a prime follows from Bertrand's postulate). Let $x_1$ and $y_1$ be $p$-cycles on $\{1, \ldots, p\}$ such that $\delta(x_1,y_1) = 3$ in $\Gamma_{\mathcal{S}}(S_p)$ (since $p \geqs 7$, the argument in the previous paragraph establishes the existence of such elements). If $n$ is odd, then let $x_2$ and $y_2$ be $(n-p)$-cycles on the remaining points $\{p+1, \ldots, n\}$. On the other hand, if $n$ is even then let $x_2$ and $y_2$ be $(n-p-1)$-cycles on  $\{p+1, \ldots, n\},$ fixing $n-1$ and $n$, respectively. Set $x = x_1x_2$ and $y = y_1y_2$ as elements of $G \setminus G_0$. 

We claim that $\delta(x,y) \geqs 3$. To see this, let $H$ be a maximal subgroup of $G$ containing $\la x, y \ra$. Note that $H$ contains a $p$-cycle. Clearly, $S_p \times S_{n-p}$ (the stabiliser in $G$ of the subset $\{1, \ldots, p\}$) is the only intransitive maximal subgroup with this property. If $H$ is primitive, then a classical theorem of Jordan implies that $H = G_0$, which is not possible since neither $x$ nor $y$ is contained in $G_0$. Similarly, $H$ is not transitive and imprimitive since it contains a $p$-cycle with $p>n/2$. So we conclude that $H = S_p \times S_{n-p}$ is the unique maximal overgroup of $\la x,y \ra$ and thus $B_1(x)$ and $B_1(y)$ are contained in $H$. As a consequence, we deduce that $\delta(x,y) \leqs 2$ in $\Gamma_{\mathcal{S}}(G)$ only if $\delta(x_1,y_1) \leqs 2$ in $\Gamma_{\mathcal{S}}(S_p)$. Therefore, our choice of $x_1,y_1 \in S_p$ implies that $\delta(x,y) \geqs 3$ and the proof is complete.
\end{proof}

For the analysis of alternating groups, it will be convenient to define the following set:
\[
\mathcal{P} = \{ p, p+1 \,:\, \mbox{$p \geqs 7$ is a prime and $p \equiv 3 \imod{4}$} \}.
\]

\begin{thm}\label{t:alt}
Let $G = A_n$ with $n \geqs 7$. 
\begin{itemize}\addtolength{\itemsep}{0.2\baselineskip}
\item[{\rm (i)}] We have $3 \leqs \delta_{\mathcal{S}}(G) \leqs 5$. 
\item[{\rm (ii)}] If $n \not\in \mathcal{P}$ then $\delta_{\mathcal{S}}(G)=3.$
\end{itemize}
In addition, if $x \in G$ is nontrivial then there exists an involution $y \in G$ with $\delta(x,y) \leqs 2$.
\end{thm}

\begin{proof}
Let $x \in G$ be a nontrivial element and write $x=x_1\cdots x_t$ as a product of disjoint cycles. Consider the abelian subgroup $H=\langle x_1,\dots,x_t\rangle.$ It is easy to check that either
\begin{itemize}\addtolength{\itemsep}{0.2\baselineskip}
\item[{\rm (a)}] $N_G(H)$ contains an involution $z$, or 
\item[{\rm (b)}] $x$ is a $q$-cycle and $n \in \{q,q+1\}$, where $q$ is a prime power with $q\equiv 3 \imod{4}$. 
\end{itemize}

If (a) holds, then $\langle x, z\rangle \leqs \langle H, z\rangle$ and we deduce that 
$x \sim z$ in $\Gamma_{\mathcal{S}}(G)$ since $\langle H, z\rangle$ is soluble. Now assume (b) holds and write $q=p^m$ with $p$ a prime. Let $P$ be a Sylow $p$-subgroup of $G$ containing $x$. If $m \geqs 2$ then $P$ is non-cyclic and thus $N_G(P)$ contains an involution $z$ by \cite[Theorem 2.1]{gnt}. But then $\langle x, z \rangle \leqs \langle P, z \rangle$ is soluble and so once again we conclude that $x$ is adjacent to an involution. It follows that if $n \not\in \mathcal{P}$, then every vertex in $\Gamma_{\mathcal{S}}(G)$ is adjacent to an involution and thus $\delta_{\mathcal{S}}(G) \leqs 3$.

Now assume $n \in \mathcal{P}$ and let $p$ be the prime in $\{n-1,n\}$. As noted above, if $x$ is not a $p$-cycle then it is adjacent to an involution. However if $x$ is a $p$-cycle, then we have $N_G(\la x \ra) = {\rm AGL}_{1}(p) \cap G = C_p{:}C_{(p-1)/2}$ and every element in $N_G(\la x \ra) \setminus \la x \ra$ is adjacent to an involution. So in this situation, we conclude that every vertex in $\Gamma_{\mathcal{S}}(G)$ has distance at most $2$ from an involution, which immediately yields the bound $\delta_{\mathcal{S}}(G) \leqs 5$. 
 
To complete the proof, we need to show that $\delta_{\mathcal{S}}(G) \geqs 3$ and we can essentially repeat the argument in the proof of Theorem \ref{t:sym}. First assume $n=p$ is a prime. Then as noted in the previous proof,  there exist $p$-cycles $x,y \in G$ such that $\delta(x,y)=3$ in $\Gamma_{\mathcal{S}}(S_n)$ and thus 
$\delta(x,y) \geqs 3$ in $\Gamma_{\mathcal{S}}(G)$.  Finally, suppose $n \geqs 8$ is composite. The cases $n \in \{8,9\}$ can be checked directly, so let us assume $n \geqs 10$. As in the proof of Theorem \ref{t:sym}, let $p$ be a prime with $n/2< p < n-2$ and fix $p$-cycles $x_1$ and $y_1$ on $\{1, \ldots, p\}$ such that $\delta(x_1,y_1) = 3$ in $\Gamma_{\mathcal{S}}(S_p)$, noting that $\delta(x_1,y_1) \geqs 3$ in $\Gamma_{\mathcal{S}}(A_p)$. Set $x=x_1x_2$ and $y=y_1y_2$ as elements of $G$, where $x_2$ and $y_2$ are $(n-p)$-cycles on the remaining points $\{p+1, \ldots, n\}$ if $n$ is even, otherwise $x_2$ and $y_2$ are $(n-p-1)$-cycles on $\{p+1, \ldots, n\},$
fixing $n-1$ and $n$, respectively. Then by repeating the argument in the proof of the previous theorem, we deduce that $\delta(x,y) \leqs 2$ in $\Gamma_{\mathcal{S}}(G)$ only if $\delta(x_1,y_1) \leqs 2$ in $\Gamma_{\mathcal{S}}(A_p)$. Therefore, our choice of $p$-cycles $x_1,y_1 \in S_p$ implies that $\delta(x,y) \geqs 3$ and the proof of the theorem is complete.
\end{proof}

\begin{lem}\label{l:alt2}
Let $G = A_{n}$, where $n=p+1$ and $p \geqs 7$ is a Mersenne prime. Then $\delta_{\mathcal{S}}(G)=3$.
\end{lem}

\begin{proof}
Write $p = 2^r-1$ and note that it suffices to show that every nontrivial $x \in G$ is adjacent in $\Gamma_{\mathcal{S}}(G)$ to an involution. By arguing as in the proof of Theorem \ref{t:alt}, we may assume $x$ is a $p$-cycle. Then $x$ is contained in a maximal subgroup $H = {\rm AGL}_{r}(2)$ and $\la V, x\ra$ is soluble, where $V = (C_2)^r$ is the socle of $H$. We conclude that $x$ is adjacent to an involution and the result follows.
\end{proof}

\begin{lem}\label{l:alt}
Let $G = A_n$, where $n \in \mathcal{P}$ and $n \leqs 60$.
\begin{itemize}\addtolength{\itemsep}{0.2\baselineskip}
\item[{\rm (i)}] If $n \in \{7,8,32\}$ then $\delta_{\mathcal{S}}(G) = 3$, whereas 
$\delta_{\mathcal{S}}(G) = 4$ if $n \in \{11,12\}$.
\item[{\rm (ii)}] If $n \not\in \{7,8,11,12,32\}$ then $\delta_{\mathcal{S}}(G) \in \{4,5\}$.
\end{itemize}
\end{lem}

\begin{proof}
By Theorem \ref{t:alt} we have $3 \leqs \delta_{\mathcal{S}}(G) \leqs 5$. 
If $n = 7$ then Proposition \ref{p:small} gives $\delta_{\mathcal{S}}(G)=3$, while Lemma \ref{l:alt2} applies if $n \in \{8,32\}$. So to complete the proof, we may assume $n \not\in \{7,8,32\}$ and it suffices to show $\delta_{\mathcal{S}}(G) \geqs 4$, with equality if $n \in \{11,12\}$. 

Write $n = p$ or $p+1$, where $p$ is a prime, and fix a $p$-cycle $x \in G$. As explained in Section \ref{ss:comp}, we can use {\sc Magma} to determine $B_1(x)$, which in each case is simply the set of nontrivial elements in $N_G(\la x \ra) = {\rm AGL}_{1}(p) \cap G$. Then by random search, we can find an element $g \in G$ such that $\la a,b \ra$ is insoluble for all $a\in B_1(x)$, $b \in B_1(x^g) = B_1(x)^g$. This immediately implies that $\delta(x,x^g) \geqs 4$ and thus $\delta_{\mathcal{S}}(G) \geqs 4$. 

Next assume $G = A_{11}$. Let $x \in G$ be nontrivial and recall from the proof of Theorem \ref{t:alt} that $x$ is adjacent to an involution unless $|x|=11$, so it suffices to show that any two $11$-cycles are connected by a path of length at most $4$. To do this, we proceed as in Example \ref{e:m24}, working with the set $Y$ of elements in $G$ of order $11$. Here $|Y| = 3628800$ and we note that $|B_1(x)| = 54$ and $|B_2(x)| = 29974$ for all $x \in Y$. We leave the reader to check the details. An entirely similar argument applies when $G = A_{12}$, noting that $|Y| = 43545600$ and $|B_1(x)| = 54$, $|B_2(x)| = 61214$ for all $x \in Y$.
\end{proof}

A prime number $p$ is a \emph{Sophie Germain prime} if $2p+1$ is also a prime number. Recall that a famous conjecture in number theory asserts that there are infinitely many such primes. If we assume the validity of this conjecture, then our next result, which coincides with Theorem \ref{t:main4}, establishes the existence of infinitely many finite simple groups $G$ with $\delta_{\mathcal{S}}(G) \geqs 4$.

\begin{thm}\label{t:sophie}
If $p \geqs 5$ is a Sophie Germain prime, then $\delta_{\mathcal{S}}(A_{2p+1}) \in \{4,5\}$.
\end{thm}

\begin{proof}
Set $G = A_q$, where $q=2p+1$ and $p \geqs 5$ is a Sophie Germain prime. It suffices to show that there exist two $q$-cycles $x$ and $y$ with $\delta(x,y) \geqs 4$. With this aim in mind, let $\mathcal{A}$ be the set of $q$-cycles in $G$ and fix an element $x \in \mathcal{A}$. We are interested in estimating the size of the set $\mathcal{A} \cap B_3(x)$, with the aim of establishing the bound $|\mathcal{A} \cap B_3(x)| < |\mathcal{A}|$. It will be helpful to write
\[
|\mathcal{A} \cap B_3(x)| = 1+\a_1 + \a_2+\a_3,
\]
where $\a_{\ell} = |\{ y \in \mathcal{A} \,:\, \delta(x,y) = \ell \}|$.

First consider $B_1(x)$. If $p \ne 5,11$ then 
\[
N_G(\la x \ra) = {\rm AGL}_{1}(q) \cap G = C_q{:}C_p
\]
is the unique maximal subgroup of $G$ containing $x$ (see the proof of \cite[Theorem 3.1]{BH}) and thus $B_1(x)$ coincides with the set of nontrivial elements in this subgroup. One can check that the same conclusion holds when $p \in \{5,11\}$. As a consequence, every element in $B_1(x)$ of order $p$ has cycle-shape $(p,p,1)$ and we let $\mathcal{B}$ be the set of all elements in $G$ of this form.

Let $y\in \mathcal{A}$. If $\delta(x,y)=1,$ then $y\in \langle x\rangle$ and thus 
\begin{equation}\label{e:al1}
\a_1 = q-2.
\end{equation}

Next suppose $\delta(x,y)=2$. Then there exists $z \in \mathcal{B}$ which normalises both $\langle x\rangle$ and $\langle y\rangle.$ Now $N_G(\la x\ra)$  contains $q$ subgroups of order $p$, each of which normalises $p(p-1)$ Sylow $q$-subgroups of $G$. This implies that 
\begin{equation}\label{e:al2}
\a_2 \leqs pq(p-1)(q-1).
\end{equation}

Finally, let us assume $\delta(x,y)=3.$ This means that there is a path $x \sim a \sim b \sim y$ in $\Gamma_{\mathcal{S}}(G)$, where $a \in N_{G}(\langle x\rangle) \cap \mathcal{B}$, $b \in N_{G}(\langle y\rangle) \cap \mathcal{B}$ and $\la a,b \ra$ is soluble. Fix $a\in \mathcal{B}$. By relabelling, we may assume that
\[
a=(1,2,\ldots, p)(p+1,p+2,\ldots,2p).
\]
We will estimate the number of subgroups $\langle b\rangle$ in $G$, where $b \in \mathcal{B}$ and $H = \langle a,b \rangle$ is soluble. There are four possibilities for the action of $H$ on $\{1, \ldots, q\}$, which we will consider in turn:
\begin{itemize}\addtolength{\itemsep}{0.2\baselineskip}
\item[{\rm (i)}] $H$ acts transitively on $\{1, \ldots, q\}$.
\item[{\rm (ii)}] $H$ has an orbit of size $2p$.
\item[{\rm (iii)}] $H$ has orbits of size $p$ and $p+1$. 
\item[{\rm (iv)}] $H$ has two orbits of size $p$.
\end{itemize}

First consider (i). Here $H = C_q{:}C_p$ is the normaliser of a Sylow $q$-subgroup of $G$. As noted above, the element $a$ normalises exactly $p(p-1)$ Sylow $q$-subgroups of $G$, and each normaliser of a Sylow $q$-subgroup contains $q$ distinct subgroups of order $p$. It follows that there are at most $pq(p-1)$ possible choices for $\langle b\rangle$.

Next let us turn to (ii), in which case we may view $H$ as a transitive group on $\{1, \ldots, 2p\}$. Since we are assuming that $H$ is soluble and $2p$ is not a prime power, it follows that $H$ is imprimitive. More precisely, since $H$ is generated by two elements of order $p$, we deduce that $H \leqs C_2\wr C_p$ preserves a partition $\L$ of $\{1, \ldots, 2p\}$ into $p$ blocks of size $2$. Since $H$ contains $a$, the set of blocks comprising $\L$ is uniquely determined by the choice of $j\in \{p+1,\ldots, 2p\}$ such that $\{1,j\}$ is a block. Therefore, there are at most $p$ choices for $\L$. And once $\L$ has been chosen, there are clearly at most $2^{p-1}$ possible choices for $\langle b\rangle.$ All together, it follows that there are at most $2^{p-1}p$ choices for $\langle b\rangle$ in case (ii).

Now suppose (iii) holds, say $X$ is an $H$-orbit of size $p$ and $Y$ is an orbit of size $p+1.$ Here the permutation group $L$ of degree $p+1$ induced by the action of $H$ on $Y$ is soluble, transititive and contains a $p$-cycle. Moreover, $L$ is primitive and by \cite[Theorem 1.2]{Jones}, this is only possible if $p+1=2^k$ for some $k \geqs 1$. But this would imply that $p=2^k-1$ and $q=2^{k+1}-1$ are both Mersenne primes, which can only happen if $p=3$ and $q=7$. Since $p \geqs 5$, we conclude that this case does not arise.
	
Finally, let us consider (iv). Given the form of $a$, it follows that the orbits of $H$ are $\{1,\ldots,p\},$  $\{p+1,\ldots,2p\}$ and $\{q\}.$ In addition, the solubility of $H$ implies that 
\[
b \in \langle (1,\dots, p), (p+1,\dots,2p)\rangle
\]
and so there are at most $p-1$ possible choices for $\langle b\rangle$.

\vs

By bringing together the above estimates, we conclude that if $p \geqs 5$ then there are at most
\[
\b=pq(p-1)+2^{p-1}p+p-1
\]
possible choices for $\la b\ra$ such that $\la a,b \ra$ is soluble. Since $\langle x \rangle$ is normalised by $q$ subgroups of order $p$, each of which normalises $p(p-1)$ Sylow $q$-subgroups, we deduce that 
\begin{equation}\label{e:al3}
\a_3 \leqs pq(p-1)(q-1)\b.
\end{equation}

Finally, by combining the bounds in \eqref{e:al1}, \eqref{e:al2} and \eqref{e:al3}, we get 
\[
|\mathcal{A} \cap B_3(x)| \leqs q-1 + pq(p-1)(q-1)(\b+1).
\]
It is routine to check that this upper bound is less than $|\mathcal{A}| = (q-1)!$ for all $p \geqs 5$. 
\end{proof}

Let $G$ be a finite group and recall that the \emph{non-generating graph} of $G$ is a graph on the nontrivial elements of $G$, where $x$ is adjacent to $y$ if $G \ne \la x,y \ra$. Note that if $G$ is insoluble and $R(G)=1$, then the soluble graph $\Gamma_{\mathcal{S}}(G)$ is a subgraph of the non-generating graph and so it is interesting to compare the diameters of these two connected graphs. 

We close this section by showing that these two diameters can be different. This relies on the following lemma.

\begin{lem}\label{l:nongen}
Let $G=A_p$ with $p \geqs 5$ a prime. Then the diameter of the non-generating graph of $G$ is at most $3$.
\end{lem}

\begin{proof}
Let $x,y \in G$ be nontrivial and let $\Delta(G)$ denote the non-generating graph of $G$. As observed in the proof of Theorem \ref{t:alt}, if neither $x$ nor $y$ is a $p$-cycle, then $x$ and $y$ are adjacent to an involution in $\Gamma_{\mathcal{S}}(G)$ and thus $\delta(x,y) \leqs 3$ in $\Delta(G)$. So let us assume $x$ is a $p$-cycle. If $y$ is also a $p$-cycle, then there exists $1 \ne z \in N_G(\la y \ra)$ with at least one fixed point on $\{1, \ldots, p\}$. Similarly, if $y$ is not a $p$-cycle, then some nontrivial power of $y$ has at least one fixed point. So in both cases, we can find a nontrivial element $z \in N_G(\la y \ra)$ that fixes a point $k \in \{1, \ldots, p\}$. Since $x$ is a $p$-cycle, there is also an element $1 \ne z' \in N_G(\la x \ra)$ fixing $k$ and we conclude that $x \sim z' \sim z \sim y$ is a path in $\Delta(G)$. The result follows.  
\end{proof}

\begin{rem}\label{r:ap}
Let $G=A_p$ with $p \geqs 13$ a prime. By the proof of \cite[Proposition 3.8]{BH}, there exist two $p$-cycles $x$ and $y$ in $G$ with the property that if $z \in G$ is any nontrivial element, then either $G = \la x,z \ra$ or $G = \la y, z\ra$. This immediately implies that $\delta(x,y) \geqs 3$ in the non-generating graph $\Delta(G)$ and therefore the diameter of $\Delta(G)$ is precisely $3$ by Lemma \ref{l:nongen}.
\end{rem}

By combining Lemma \ref{l:nongen} with Theorem \ref{t:sophie}, we obtain the following corollary.

\begin{cor}\label{c:nongen}
There exist finite simple groups $G$ such that $\delta_{\mathcal{S}}(G)$ is strictly larger than the diameter of the non-generating graph of $G$.
\end{cor}

\section{Groups of Lie type}\label{s:lie}

In this section we complete the proofs of Theorems \ref{t:main2} and \ref{t:main3} by handling the case where $G$ is an almost simple group of Lie type over $\mathbb{F}_q$ with socle $G_0$. The two-dimensional linear groups with $G_0 = {\rm L}_2(q)$ merit special attention and we deal with them separately in Section \ref{ss:psl2}.

\subsection{Two-dimensional linear groups}\label{ss:psl2}

Here we determine the precise diameter of $\Gamma_{\mathcal{S}}(G)$ when $G$ is an almost simple group with socle $G_0 = {\rm L}_2(q)$. 

\begin{thm}\label{t:psl2}
Let $G$ be an almost simple group with socle $G_0 = {\rm L}_{2}(q)$. Then 
\[
\delta_{\mathcal{S}}(G) = \left\{\begin{array}{ll}
2 & \mbox{if ${\rm PGL}_{2}(q) \leqs G$ or $q \in \{5,7\}$} \\
3 & \mbox{otherwise.} 
\end{array}\right.
\]
\end{thm}

\begin{proof}
The groups with $q \leqs 11$ can be checked directly using {\sc Magma}, so we will assume $q \geqs 13$. Let $A$ and $B$ be maximal subgroups of $G$ of type ${\rm GL}_{1}(q^2)$ and $P_1$, respectively, so $A$ is the normaliser of a nonsplit maximal torus of $G_0$ and $B$ is a Borel subgroup. Note that $A$ and $B$ are soluble and we have
\begin{equation}\label{e:cover}
G = \bigcup_{g \in G} A^g \cup \bigcup_{g \in G} B^g
\end{equation}
(see \cite[Corollary 4.3]{BL}, for example). 

First assume ${\rm PGL}_{2}(q) \leqs G$. Here we observe that the intersection of any two subgroups in the union \eqref{e:cover} is nontrivial. Indeed, we have $|B|^2 >|G|$ and $|A||B| >|G|$, while any two conjugates of $A$ intersect nontrivially since $b(G,A) = 3$ (see \cite[Lemma 4.8]{B20}). This immediately implies that if $x,y \in G$ are nontrivial, then $B_1(x) \cap B_1(y)$ is nonempty and thus $\delta_{\mathcal{S}}(G) = 2$ as required.

For the remainder, we may assume $q$ is odd and $G \cap {\rm PGL}_{2}(q) = G_0$. Let $x \in G_0$ be an element of order $(q+1)/2$, in which case $H = N_G(\la x \ra)$ is the unique soluble maximal subgroup of $G$ containing $x$ and thus $B_1(x) \subseteq H$. Now $H$ is conjugate to $A$ and we note that $b(G,A) = 2$ by \cite[Lemma 4.8]{B20}. Therefore, $\delta_{\mathcal{S}}(G) \geqs 3$ by Lemma \ref{l:base}. 

It remains to verify the bound $\delta_{\mathcal{S}}(G) \leqs 3$. If $q \equiv 1 \imod{4}$ then $|A|$ and $|B|$ are even, so every nontrivial element is adjacent to an involution and the result follows. However, if $q \equiv 3 \imod{4}$ then $|B|$ is odd and some additional argument is required. Let $x,y \in G$ be nontrivial and note that every involution in $G$ is contained in $G_0$. If both $x$ and $y$ are contained in conjugates of $A$, then since $|A|$ is even (and soluble), there exist involutions $z_1$ and $z_2$ such that $x \sim z_1 \sim z_2 \sim y$ is a path in $\Gamma_{\mathcal{S}}(G)$. Similarly, if $x$ and $y$ are both contained in conjugates of $B$, say $B_1$ and $B_2$, then $\delta(x,y) \leqs 2$ since $B_1 \cap B_2 \ne 1$. Finally, suppose $x \in A$ and $y \in B$. Let $u \in B \cap G_0$ be an element of order $(q-1)/2$ and let $z \in A$ and $w \in N_G(\la u \ra)$ be involutions (note that $N_{G_0}(\la u \ra) = D_{q-1}$). Now $B \cap G_0$ acts regularly (by conjugation) on the set of involutions in $G$, so $z = w^b$ for some $b \in B$ and thus $z \in N_G(\la u^b \ra)$. Therefore, we have a path $x \sim z \sim u^b \sim y$ of length $3$ and the proof is complete.
\end{proof}

\begin{cor}\label{c:psl2}
There are infinitely many simple groups with $\delta_{\mathcal{S}}(G) = 2$.
\end{cor}

We can also establish Theorem \ref{t:main3} for the groups with socle ${\rm L}_2(q)$.

\begin{prop}\label{p:psl20}
Let $G$ be an almost simple group with socle $G_0 = {\rm L}_{2}(q)$ and let $x \in G$ be nontrivial. Then there exists an involution $y \in G$ with $\delta(x,y) \leqs 2$.
\end{prop}

\begin{proof}
Define $A$ and $B$ as in the proof of Theorem \ref{t:psl2} and recall that we may assume $x$ is contained in $A$ or $B$ (see \eqref{e:cover}). Here $A$ is soluble with even order and thus every element in $A$ is adjacent to an involution. The same conclusion holds if $x \in B$ and $q \not\equiv 3 \imod{4}$. Finally, suppose $x \in B$ and $q \equiv 3 \imod{4}$. Fix $z \in B$ of order $(q-1)/2$ and note that $|N_G(\la z \ra)|$ is even. Therefore, if $y \in N_G(\la z \ra)$ is an involution, then $x \sim z \sim y$ and the result follows. 
\end{proof}

\subsection{Exceptional groups}\label{ss:excep}

For the remainder of Section \ref{s:lie}, we will assume $G_0$ is a finite simple group of Lie type over $\mathbb{F}_q$. We begin by fixing some standard notation. 

Write $G_0 = (\bar{G}_{\s})'$, where $\bar{G}$ is a simple algebraic group of adjoint type defined over the algebraic closure of $\mathbb{F}_q$ and $\s$ is an appropriate Steinberg endomorphism of $\bar{G}$. Let $\widetilde{G} = \bar{G}_{\s}= {\rm Inndiag}(G_0)$ be the subgroup of ${\rm Aut}(G_0)$ generated by the inner and diagonal automorphisms of $G_0$. In addition, write $q=p^f$ with $p$ a prime. 

Recall that a subgroup $B$ of $G_0$ is a Borel subgroup if $B =N_{G_0}(P)$, where $P$ is a Sylow $p$-subgroup of $G_0$. The following elementary lemma will be useful. 

\begin{lem}\label{l:borel}
Let $G_0$ be a finite simple group of Lie type over $\mathbb{F}_q$ and let $B$ be a Borel subgroup of $G_0$. Then $|B|$ is odd if and only if $G_0$ is isomorphic to ${\rm L}_{2}(q)$ with $q \equiv 3 \imod{4}$.
\end{lem}

\begin{proof}
We may assume $q$ is odd. If $G_0 = {\rm L}_2(q)$ then $|B| = q(q-1)/2$ and thus $|B|$ is odd if and only if $q \equiv 3 \imod{4}$. If $G_0 = {\rm L}_n^{\e}(q)$ with $n \geqs 3$, then
\[
|B| = \left\{\begin{array}{ll}
\frac{1}{d}q^{n(n-1)/2}(q-1)^{n-1} & \mbox{if $\e=+$} \\
\frac{1}{d}q^{n(n-1)/2}(q^2-1)^r & \mbox{if $\e=-$,}
\end{array}\right.
\]
where $d=(n,q-\e)$ and $r \geqs 1$ is the twisted Lie rank of ${\rm U}_n(q)$. In both cases we see that $|B|$ is even and it is straightforward to check that the same conclusion holds in all the remaining cases. We omit the details.
\end{proof}

For the remainder of Section \ref{ss:excep}, we will assume $G_0$ is a simple exceptional group of Lie type. The following result is a key tool in our proof of Theorem \ref{t:main1}. Recall that a prime divisor $r$ of $q^m-1$ is a \emph{primitive prime divisor} if $q^i-1$ is indivisible by $r$ for all $1 \leqs i < m$.

\begin{prop}\label{p:ex_real}
Let $G_0$ be a simple exceptional group of Lie type and let $x \in G$ be an element of prime order $r$. Then either
\begin{itemize}\addtolength{\itemsep}{0.2\baselineskip}
\item[{\rm (i)}] $x$ is adjacent in $\Gamma_{\mathcal{S}}(G)$ to an involution; or
\item[{\rm (ii)}] $G_0 = E_6^{\e}(q)$ and $r$ is a primitive prime divisor of $q^{9a}-1$, where $a=\frac{1}{2}(3-\e)$. 
\end{itemize}
\end{prop}

\begin{proof}
Notice that (i) holds if $|C_{G_0}(x)|$ is even or $x$ is real, so we may assume $r$ is odd. In particular, we observe that (i) holds if $C_{G_0}(x)$ is insoluble.

First assume $x \in G \setminus \widetilde{G}$. Then either $q = q_0^r$ and $x$ is a field automorphism, or $r=3$, $G_0 = {}^3D_4(q)$ and $x$ is a graph automorphism. In both cases, it is easy to see that $|C_{G_0}(x)|$ is even and thus $x$ is adjacent to an involution. For example, if $G_0 = F_4(q)$ and $x$ is a field automorphism, then $C_{G_0}(x) = F_4(q_0)$. Similarly, if $G_0 = {}^3D_4(q)$ and $x$ is a graph automorphism, then 
\begin{equation}\label{e:cent}
C_{G_0}(x) \in \{G_2(q), {\rm PGL}_{3}^{\e}(q), [q^5].{\rm SL}_{2}(q)\}
\end{equation}
and the claim follows.   

For the remainder, we may assume $x \in \widetilde{G}$ is unipotent or semisimple. 
First assume $x$ is unipotent, so $r=p$ and $G_0 \ne {}^2B_2(q), {}^2F_4(q)'$. Let $P$ be a Sylow $p$-subgroup of $G_0$ containing $x$ and consider the Borel subgroup $B = N_{G_0}(P)$. Then $B$ is soluble and it contains an involution $z$ by Lemma \ref{l:borel}. Therefore, $\la x, z\ra \leqs B$ is soluble and thus $x \sim z$.

Finally, let us assume $x$ is semisimple. Let $W = N_{\bar{G}}(\bar{T})/\bar{T}$ be the Weyl group of $\bar{G}$, where $\bar{T}$ is a $\s$-stable maximal torus  of $\bar{G}$ containing $x$. If $G_0 \ne E_6^{\e}(q)$, then $W$ contains a central involution, which acts by inversion on every maximal torus of $\widetilde{G}$. Therefore, every semisimple element in $\widetilde{G}$ is real.

So to complete the proof, we may assume $G_0 = E_6^{\e}(q)$. If $x$ is non-regular, then it is easy to see that $|C_{G_0}(x)|$ is even, so we may assume $x$ is a regular semisimple element that is contained in a unique maximal torus $T$ of $\widetilde{G}$. In addition, we can assume $|T|$ is odd. 

Recall that there is a natural action of $\s$ on $W$ and the $\s$-class of $s \in W$ is defined to be the subset $\{w^{\s}sw^{-1} \, : \, w \in W\}$ (so if $\s$ acts trivially on $W$, then this coincides with the usual conjugacy class of $s$ in $W$). As a consequence of the Lang-Steinberg theorem, there is a bijection from the set of $\widetilde{G}$-classes of maximal tori in $\widetilde{G}$ to the set of $\s$-classes in $W$ (see \cite[Chapter 25]{MT}, for example). Moreover, if $T$ corresponds to the $\s$-class of $s \in W$, then  
\begin{equation}\label{e:ngt}
|N_{G_0}(T)| = |T_0||C_W(s)|
\end{equation}
with $T_0 = T \cap G_0$. As a consequence, if $|C_W(s)|$ is even then there exists an involution $z \in N_{G_0}(T)$ and we deduce that $x \sim z$ since $\la x, z\ra$ is contained in the soluble subgroup $\la T, z\ra \leqs \widetilde{G}$. Now $W$ has $25$ conjugacy classes and one can check that $|C_W(s)|$ is even unless $s$ has order $9$ (there is a unique class of such elements), in which case   $C_W(s) = \la s \ra$. In terms of the above bijection, the $\s$-class of $s$ corresponds to a cyclic torus $T$ of order $q^6 + \e q^3 +1$ and it follows that $r$ is a primitive prime divisor of $q^{9a}-1$, where $a = \frac{1}{2}(3-\e)$. This is the special case arising in part (ii) of the proposition and so the proof is complete.
\end{proof}

\begin{rem}\label{r:exsp}
The special case highlighted in part (ii) of Proposition \ref{p:ex_real} is a genuine exception. For instance, let us assume $G = \tilde{G} = G_0 = E_6(q)$ and $x \in G$ has order $q^6+q^3+1$ (for example, we could take $q=2$, in which case $x$ has order $73$). By Weigel \cite[Section 4(g)]{W}, $x$ is contained in a unique maximal subgroup of $G$, namely $H =  {\rm L}_{3}(q^3).3$, so $B_1(x)$ is a subset of $H$. Moreover, by considering the subgroups of $H$, we deduce that $x$ is contained in a unique maximal soluble subgroup of $H$, namely $L = N_G(\la x \ra) = C_{q^6+q^3+1}{:}9$. This implies that $B_1(x)$ coincides with the set of nontrivial elements in $L$ and thus $x$ is not adjacent to an involution.
\end{rem}

\begin{cor}\label{c:ex}
Let $G$ be an almost simple group with socle $G_0$, an exceptional group of Lie type. Then for all nontrivial $x \in G$, there exists an involution $y \in G$ with $\delta(x,y) \leqs 2$. In particular, $\delta_{\mathcal{S}}(G) \leqs 5$.
\end{cor}

\begin{proof}
Let $x \in G$ be nontrivial and fix $m \in \mathbb{N}$ so that $z = x^m$ has prime order $r$. Then by applying Proposition \ref{p:ex_real}, we may assume $G_0 = E_6^{\e}(q)$, $z$ is semisimple and $r$ is a primitive prime divisor of $q^{9a}-1$, where $a = \frac{1}{2}(3-\e)$. Here $T = C_{\widetilde{G}}(z)$ is a cyclic maximal torus of order $q^6+\e q^3+1$. Since $x$ commutes with $z$, we have $x \in N_G(\la z \ra)$, which is a soluble group of order divisible by $3$ (note that $N_{G_0}(T_0) \leqs N_G(\la z \ra)$ and $N_{G_0}(T_0) = T_0{:}9$, where $T_0 = T \cap G_0$). Therefore, $x$ is adjacent to an element of order $3$ in $G_0$, which in turn is adjacent to an involution by Proposition \ref{p:ex_real}.
\end{proof}

Although we do not know if the upper bound on $\delta_{\mathcal{S}}(G)$ in Corollary \ref{c:ex} is tight, we can demonstrate the existence of an exceptional group of Lie type with $\delta_{\mathcal{S}}(G) \geqs 4$. 

\begin{lem}\label{l:e62}
If $G = E_6(2)$ then $\delta_{\mathcal{S}}(G) \in \{4,5\}$.
\end{lem}

\begin{proof}
We apply a computational approach, working with {\sc Magma} \cite{magma} and a permutation representation of $G$ of degree $279006$. Let $x \in G$ be an element of order $73 = 2^6+2^3+1$ and recall that $B_1(x)$ coincides with the set of nontrivial elements in $N_G(\la x \ra) =  73{:}9$ (see Remark \ref{r:exsp}). By implementing a  random search, we can find an element $g \in G$ such that $\la a,b \ra$ is insoluble for all $a \in B_1(x)$, $b \in B_1(x^g)$. This implies that $\delta(x,x^g) \geqs 4$ and the result follows. 
\end{proof}

Finally, we complete the proof of the claim in part (ii)(d) of Theorem \ref{t:main2}. 

\begin{prop}\label{p:exdiam}
Let $G$ be a finite simple exceptional group of Lie type. Then either 
\begin{itemize}\addtolength{\itemsep}{0.2\baselineskip}
\item[{\rm (i)}] $G = {}^2G_2(3)' \cong {\rm L}_2(8)$ and $\delta_{\mathcal{S}}(G) = 2$; or
\item[{\rm (ii)}] $\delta_{\mathcal{S}}(G) \geqs 3$.
\end{itemize}
\end{prop}

\begin{proof}
If $G = {}^2G_2(3)'$ then we can appeal to Theorem \ref{t:psl2}, so for the remainder we may assume $G \ne {}^2G_2(3)'$. Following \cite{BH}, let $\gamma_u(G)$ be the uniform domination number of $G$. This is defined to be the minimal size of a set of conjugate elements $\{x_1, \ldots, x_k\}$ such that for all nontrivial $y \in G$, there exists $i \in \{1, \ldots, k\}$ such that $G = \la x_i,y\ra$. Note that $\delta_{\mathcal{S}}(G) \geqs 3$ if $\gamma_u(G) = 2$. Indeed, if $\{x_1,x_2\}$ has the given property, then $G = \la x_1, x_2\ra$ and there is no nontrivial element $y \in G$ with $x_1 \sim y \sim x_2$ because either $\la x_1, y\ra$ or $\la x_2,y\ra$ is equal to $G$. Therefore, in view of \cite[Theorem 4]{BH2}, we immediately deduce that $\delta_{\mathcal{S}}(G) \geqs 3$ if $G \ne {}^2F_4(2)', F_4(q), G_2(q)'$, so it remains to handle these three special cases.

The group $G = {}^2F_4(2)'$ can be handled using {\sc Magma}. If $x \in G$ has order $13$ then $|B_1(x)| = 77$ and it is easy to find an element $g \in G$ such that $B_1(x) \cap B_1(x^g)$ is empty, which implies that $\delta(x,x^g) \geqs 3$.

Next assume $G = G_2(q)'$. If $q=2$ then $G$ is isomorphic to ${\rm U}_{3}(3)$ and we note that $\delta_{\mathcal{S}}(G) = 3$ (for example, see the proof of Proposition \ref{p:small}). Now assume $q \geqs 3$ and let $x \in G$ be an element of order $q^2-q+1$, so $T = \la x \ra$ is a maximal torus of $G$. If $q \in \{3,4\}$ then we can use {\sc Magma} to establish the existence of an element $g \in G$ such that $\delta(x,x^g) \geqs 3$. For $q \geqs 5$, the overgroups of $T$ are described by Weigel \cite[Section 4(d)]{W} and we deduce that $B_1(x)$ coincides with the set of nontrivial elements in $L = N_G(T) = T.6$. Since $b(G,L) = 2$ by \cite{BT}, we deduce that $\delta_{\mathcal{S}}(G) \geqs 3$ via Lemma \ref{l:base}.

Finally, suppose $G = F_4(q)$ and let $x \in G$ be an element of order $q^4-q^2+1$. If $q \geqs 4$, then \cite[Section 4(f)]{W} implies that $B_1(x)$ is the set of nontrivial elements in $L = N_G(T) = T.12$, where $T = \la x \ra$. Since $b(G,L) = 2$ (see \cite{BT}), the result now follows as in the previous case. The same argument also goes through when $q=3$. Indeed, as noted in \cite[Table IV]{GK}, $x$ is contained in a unique maximal subgroup of $G$, namely $H = {}^3D_4(3).3$. In turn, $L = N_H(T) = T.12$ is the unique maximal subgroup of $H$ containing $x$ and the desired result follows via \cite{BT}. Finally, let us assume $q=2$. Here we take $T = \la x \ra$ to be a maximal torus of order $q^4+1 = 17$. Then as recorded in \cite[Table IV]{GK}, we see that $x$ is contained in precisely two maximal subgroups of $G$, which are representatives of the two conjugacy classes of subgroups isomorphic to ${\rm Sp}_8(2)$ in $G$. By working in ${\rm Sp}_{8}(2)$, we find that $B_1(x)$ coincides with the set of nontrivial elements in $L = N_G(T) = T.8$ and once again we conclude by applying \cite{BT}, which gives $b(G,L) = 2$. 
\end{proof}

\subsection{Classical groups}\label{ss:class}

Now let us assume $G$ is an almost simple classical group over $\mathbb{F}_q$ with socle $G_0$. In view of Theorem \ref{t:psl2}, we may assume $G_0$ is not isomorphic to a $2$-dimensional linear group. First we establish an analogue of Proposition \ref{p:ex_real}. 

\begin{prop}\label{p:class_real}
Let $G_0 \ne {\rm L}_2(q)$ be a simple classical group and let $x \in G$ be an element of prime order $r$. Then one of the following holds:
\begin{itemize}\addtolength{\itemsep}{0.2\baselineskip}
\item[{\rm (i)}] $x$ is adjacent in $\Gamma_{\mathcal{S}}(G)$ to an involution.
\item[{\rm (ii)}] $G_0 = {\rm L}_{n}^{\e}(q)$, $n$ is odd and $r$ is a primitive prime divisor of $q^{an}-1$, where $a = \frac{1}{2}(3-\e)$.
\item[{\rm (iii)}] $G_0 = {\rm U}_{n}(q)$, $n$ is even and $r$ is a primitive prime divisor of $q^{2(n-1)}-1$. 
\item[{\rm (iv)}] $G_0 = {\rm P\O}_{n}^{-}(q)$, $n \equiv 2 \imod{4}$, $n \geqs 10$ and $r$ is a primitive prime divisor of $q^{n}-1$.
\end{itemize}
\end{prop}

\begin{proof}
We may assume $r$ is odd. If $x \in G \setminus \widetilde{G}$ then either $q=q_0^r$ and $x$ is a field automorphism, or $G_0 = {\rm P\O}_{8}^{+}(q)$, $r=3$ and $x$ is a triality graph or graph-field automorphism. In each of these cases, it is easy to check that $|C_{G_0}(x)|$ is even and thus (i) holds. For example, if $G_0 = {\rm P\O}_{8}^{+}(q)$ and $x$ is a triality graph automorphism, then the possibilities for $C_{G_0}(x)$ are given in \eqref{e:cent}. Similarly, if $x$ is a triality graph-field automorphism then $q=q_0^3$ and $C_{G_0}(x) = {}^3D_4(q_0)$.

Now assume $x \in \widetilde{G}$. If $r=p$ then we can repeat the argument from the proof of Proposition \ref{p:ex_real}, noting that the Borel subgroups of $G_0$ have even order by Lemma \ref{l:borel} (recall that we are assuming $G_0 \ne {\rm L}_{2}(q)$). For the remainder, we may assume $r \ne p$. If we exclude the following two cases:
\begin{itemize}\addtolength{\itemsep}{0.2\baselineskip}
\item[{\rm (a)}] $G_0 = {\rm L}_{n}^{\e}(q)$ with $n \geqs 3$; and
\item[{\rm (b)}] $G_0 = {\rm P\O}_{n}^{\e}(q)$ with $n \equiv 2 \imod{4}$ and $n \geqs 10$,
\end{itemize}
then the Weyl group of $\bar{G}$ contains a central involution and by arguing as in the proof of Proposition \ref{p:ex_real} we see that every semisimple element in $\widetilde{G}$ is real. This leaves us to handle cases (a) and (b). If $x$ is non-regular then $|C_{G_0}(x)|$ is even, so we may assume $x$ is regular and is therefore contained in a unique maximal torus $T = \bar{T}_{\s}$ of $\widetilde{G}$, where $\bar{T}$ is a $\s$-stable maximal torus of $\bar{G}$. Let $W = N_{\bar{G}}(\bar{T})/\bar{T}$ be the corresponding Weyl group. As in the proof of Proposition \ref{p:ex_real}, let us assume $T$ corresponds to the $\s$-class of $s \in W$, in which case \eqref{e:ngt} holds with $T_0 = T \cap G_0$. Note that (i) holds if $|N_{G_0}(T)|$ is even, so we may assume $|N_{G_0}(T)|$ is odd. We consider cases (a) and (b) separately.

Suppose $G_0 = {\rm L}_{n}^{\e}(q)$ with $n \geqs 3$ and note that $W = S_n$. The conjugacy classes in $W$ are parameterised by partitions of $n$, so we may assume $s$ corresponds to the partition $\l = (n^{a_n}, \ldots, 2^{a_2},1^{a_1})$, where $a_{\ell}$ denotes the multiplicity of $\ell$ in the partition. Then 
\[
|T| = (q-\e)^{-1}\prod_{\ell=1}^{n} (q^{\ell}-\e^{\ell})^{a_{\ell}}, \;\; 
|C_W(s)| = \prod_{\ell=1}^n a_{\ell}!\ell^{a_{\ell}}
\]
and thus $|C_W(s)|$ is odd if and only if $\l$ consists of distinct odd parts. Since $x \in T$ is regular and has prime order, we deduce that either $n$ is odd and $\l = (n)$, or $n$ is even and $\l = (n-1,1)$. In other words, either $n$ is odd and $r$ is a primitive prime divisor of $q^{an}-1$, or $n$ is even and $r$ is a primitive prime divisor of $q^{a(n-1)}-1$, where $a = \frac{1}{2}(3-\e)$. In particular, (ii) holds if $n$ is odd. Similarly, (iii) holds if $n$ is even and $\e=-$, so we may assume $n$ is even and $\e=+$. Here $\dim C_V(x) = 1$, where $V$ is the natural module, and so we may embed $x$ in a Levi factor $L$ of a maximal parabolic subgroup $H = QL$ of $G_0$ (the stabiliser of a $1$-dimensional subspace of $V$), where the unipotent radical $Q$ is elementary abelian of order $q^{n-1}$ and $L$ is a subgroup of ${\rm GL}_{n-1}(q)$ of index $(n,q-1)$. Now $q$ is even since $|T|$ is odd and the subgroup $\la Q, x\ra$ of $H$ is soluble. In particular, $x$ is adjacent to every involution in $Q$. 

Finally, let us assume $G_0 = {\rm P\O}_{n}^{\e}(q)$ with $n \equiv 2 \imod{4}$ and $n \geqs 10$. Here $W$ is an index-two subgroup of $S_2 \wr S_{n/2}$ and we find that $|C_W(s)|$ is odd if and only if $s$ corresponds to an $\frac{n}{2}$-cycle in $S_{n/2}$. In this situation, $T$ is a cyclic torus of order $q^{n/2}-\e$ and $r$ is a primitive prime divisor of $q^{an/2}-1$, where $a = \frac{1}{2}(3-\e)$. In particular, $q$ is even. If $\e=+$ then we may embed $x$ in a Levi factor $L$ of a maximal parabolic subgroup $H = QL$ of $G_0$, where the unipotent radical $Q$ is elementary abelian of order $q^{n(n-2)/8}$ and $L = {\rm GL}_{n/2}(q)$ (here $H$ is the stabiliser of a maximal totally singular subspace of the natural module for $G_0$). Therefore, $x$ is contained in the soluble subgroup $\la Q, x \ra$ and thus $x$ is adjacent to every involution in $Q$. This leaves the case $\e=-$, which is recorded in part (iv).
\end{proof}

\begin{rem}
There are genuine exceptions that arise under the conditions recorded in parts (ii), (iii) and (iv) of Proposition \ref{p:class_real}. For example, if $G = {\rm L}_{3}(4)$ and $x \in G$ has order $7$, then $B_1(x)$ coincides with the nontrivial elements in $N_G(\la x \ra) = 7{:}3$ and thus $x$ is not adjacent to an involution. Similarly, if $G = {\rm U}_3(4)$ and $|x|=13$ then $B_1(x) \subseteq 13{:}3$ and the same conclusion holds. If $G = {\rm U}_6(2)$ and $|x|=11$ then $B_1(x) \subseteq 11{:}5$. Similarly, if $G = \O_{10}^{-}(2)$ and $|x|=11$ then $B_1(x) \subseteq 11{:}15$.
\end{rem}

We are now in a position to complete the proof of Theorem \ref{t:main3}.

\begin{cor}\label{c:class}
Let $G$ be an almost simple classical group with socle $G_0$. Then for all nontrivial $x \in G$, there exists an involution $y \in G$ with $\delta(x,y) \leqs 2$. In particular, $\delta_{\mathcal{S}}(G) \leqs 5$.
\end{cor}

\begin{proof}
Let $x \in G$ be nontrivial and assume $z = x^m$ has prime order $r$. In view of Theorem \ref{t:psl2} (and Proposition \ref{p:psl20}), we may assume that $(G_0,r)$ is one of the cases arising in parts (ii)-(iv) of Proposition \ref{p:class_real}. In each case, $z \in \widetilde{G}$ is a regular semisimple element and $C_{\widetilde{G}}(z) = T$ is a cyclic maximal torus of $\widetilde{G}$ of odd order. In particular, $q$ is even and \eqref{e:ngt} holds, where $T_0 = T \cap G_0$ and $|C_W(s)|$ is odd.

First consider case (ii) in Proposition \ref{p:class_real}, so $G_0 = {\rm L}_{n}^{\e}(q)$, $n$ is odd and $r$ is a primitive prime divisor of $q^{an}-1$, where $a = \frac{1}{2}(3-\e)$. Note that $r \geqs 2n+1$ and $|C_W(s)| = n$. Now $N_{G_0}(T_0) = T_0{:}n$ is a subgroup of $N_G(\la z \ra)$, which in turn is a soluble group containing $x$. Therefore, $x$ is adjacent to an element $z' \in G_0$ of order $r'$, where $r'$ is a prime divisor of $n$. Since $r' \leqs n$, Proposition \ref{p:class_real} implies that $z'$ is adjacent to an involution and we conclude that $x$ has distance at most $2$ from an involution, as required.

A very similar argument applies in cases (iii) and (iv), noting that $|C_W(s)| = n-1$ in (iii) and $|C_W(s)| = n/2$ in (iv). In both cases, we find that $x$ is adjacent to an element $z' \in G_0$ of order $r'$, where $r'$ is a prime divisor of $|C_W(s)|$, and the result follows.
\end{proof}

The following lemma shows that there exist finite simple classical groups with $\delta_{\mathcal{S}}(G) = 4$. 

\begin{lem}\label{l:l52}
If $G = {\rm L}_{5}^{\e}(2)$ then $\delta_{\mathcal{S}}(G) = 4$.
\end{lem}

\begin{proof}
We use {\sc Magma} to verify the result, working with the standard permutation representations of degree $31$ (for $\e=+$) and $165$ (for $\e=-$). First assume $\e=+$. Let $x \in G$ be an element of order $31$. Then $B_1(x)$ is the set of nontrivial elements in $H = N_G(\la x \ra) = 31{:}5$ and by random search we can find $g \in G$ such that $\la a, b\ra$ is insoluble for all $a \in B_1(x)$, $b \in B_1(x^g)$. This implies that $\delta(x,x^g) \geqs 4$ and thus $\delta_{\mathcal{S}}(G) \geqs 4$. An entirely similar argument applies when $\e=-$, noting that if $|x|=11$ then $B_1(x)$ is the set of nontrivial elements in $N_G(\la x \ra) = 11{:}5$.

To complete the proof, we need to show that $\delta_{\mathcal{S}}(G) \leqs 4$. First  assume $\e=+$ and observe that if $y \in G$ is nontrivial, then $|N_G(\la y\ra)|$ is even unless $|y|=31$. Therefore, if we fix an element $x \in G$ of order $31$, then it suffices to show that $\delta(x,y) \leqs 4$ for all $y \in Y$, where $Y$ is the set of elements of order $31$ in $G$. Here $|Y| = 1612800$, $|B_1(x)| = 154$ and using {\sc Magma} we compute $|B_2(x)| = 106484$. We now implement the process described in Example \ref{e:m24}, which allows us to conclude that $Y \subseteq B_4(x)$ as required. 

We can apply a very similar argument when $\e=-$. First we observe that $|N_G(\la x \ra)|$ is even unless $|x| \in \{9,11\}$. If $|x|=9$ then $x$ is contained in a soluble maximal subgroup of type ${\rm GU}_3(2) \times {\rm GU}_2(2)$ (the stabiliser of a $2$-dimensional nondegenerate subspace of the natural module for $G$) and we deduce that $x$ is adjacent to an involution in $\Gamma_{\mathcal{S}}(G)$. Therefore, it remains to show that any two elements of order $11$ are connected by a path of length at most $4$ and we proceed as above, following the procedure presented in Example \ref{e:m24}. Here we work with the set $Y$ of elements of order $11$ in $G$, noting that $|Y|=2488320$ and we have $|B_1(x)| = 54$ and $|B_2(x)| =220549$ for all $x \in Y$.
\end{proof}

We also record the following observation, which completes the proof of Theorem \ref{t:main2}.

\begin{lem}\label{l:l72}
If $G = {\rm L}_{7}^{\e}(2)$ then $\delta_{\mathcal{S}}(G) \in \{4,5\}$.
\end{lem}

\begin{proof}
By Corollary \ref{c:class} we have $\delta_{\mathcal{S}}(G) \leqs 5$. 
First assume $\e=+$ and let $x \in G$ be an element of order $127$. Then $B_1(x)$ is the set of nontrivial elements in $N_G(\la x \ra) = 127{:}7$ and by random search we can find an element $g \in G$ such that $\la a,b \ra$ is insoluble for all $a \in B_1(x)$, $b \in B_1(x^g)$. This implies that $\delta(x,x^g) \geqs 4$. An entirely similar argument applies when $\e=-$, noting that if $|x|=43$ then $B_1(x)$ is the set of nontrivial elements in $N_G(\la x \ra) = 43{:}7$.
\end{proof}

We close this section by considering the problem highlighted in Remark \ref{r:1} in Section \ref{s:intro}. Let $G$ be a non-abelian finite simple group and recall that we have shown that $\delta_{\mathcal{S}}(G) \geqs 3$ if $G$ is not isomorphic to a classical group (this is part (ii)(d) in Theorem \ref{t:main2}). For classical groups, we have  observed that $\delta_{\mathcal{S}}(G) = 2$ if $G = {\rm L}_{2}(q)$ with $q$ even (or $q=7$) or if $G \in \{{\rm L}_{3}(2), {\rm U}_4(2)\}$. We are not aware of any additional simple groups with $\delta_{\mathcal{S}}(G) = 2$, but a complete classification remains out of reach. 

In the next result, we classify the linear groups with $\delta_{\mathcal{S}}(G) = 2$.

\begin{prop}\label{p:class_new}
Suppose $G = {\rm L}_{n}(q)$ is a simple group. Then $\delta_{\mathcal{S}}(G) = 2$ if and only if $n=2$ and $q$ is even or $q \in \{5,7\}$, or $(n,q) = (3,2)$.   
\end{prop}

\begin{proof}
For $n=2$ we refer the reader to Theorem \ref{t:psl2}. Next assume $n \geqs 3$ is odd. If $(n,q) = (3,2)$ then $G \cong {\rm L}_{2}(7)$ and $\delta_{\mathcal{S}}(G) = 2$ by Theorem \ref{t:psl2}. Similarly, the proof of Proposition \ref{p:small} yields $\delta_{\mathcal{S}}(G) = 3$ if $(n,q) = (3,4)$. In each of the remaining cases, \cite[Theorem 6(ii)]{BH2} states that the uniform domination number of $G$ is equal to $2$ and we deduce that $\delta_{\mathcal{S}}(G) \geqs 3$ as in the proof of Proposition \ref{p:exdiam}.

Now assume $n \geqs 4$ is even. Let $x \in G$ be a Singer element, so $T = \la x \ra$ is a maximal torus of $G$ of order $(q^n-1)/d(q-1)$, where $d = (n,q-1)$. Then $x$ is contained in a unique maximal soluble subgroup of $G$, namely $H = N_G(T) = T.n$, and thus $B_1(x)$ coincides with the set of nontrivial elements in $H$ (for example, this is an easy consequence of Suprunenko's structure theory of primitive maximal soluble subgroups of linear groups, see \cite{Sup}). By \cite{BT} we have $b(G,H) = 2$ and thus 
$\delta_{\mathcal{S}}(G) \geqs 3$ via Lemma \ref{l:base}.
\end{proof}

We can also easily eliminate odd-dimensional unitary groups.

\begin{prop}\label{p:class_new2}
Let $G = {\rm U}_{n}(q)$, where $n \geqs 3$ is odd. Then $\delta_{\mathcal{S}}(G) \geqs 3$. 
\end{prop}

\begin{proof}
If $(n,q) = (3,3)$ then $\delta_{\mathcal{S}}(G) = 3$ by the proof of Proposition \ref{p:small}. Similarly, the case $(n,q) = (3,5)$ can be checked using {\sc Magma}. In each of the remaining cases, the uniform domination number of $G$ is equal to $2$ (see \cite[Theorem 6(ii)]{BH2}) and the result follows as in the proof of the previous proposition.
\end{proof}

\section{Some related graphs}\label{s:gen}

Recall that the soluble graph of a finite group $G$ is a natural generalisation of the widely studied commuting graph. Indeed, the soluble graph encodes the pairs of elements that generate a soluble subgroup of $G$, whereas the commuting graph is concerned with the pairs generating an abelian group. Of course, there are many natural families of groups that lie between soluble and abelian, including the supersoluble, nilpotent, metabelian and metacyclic groups. For each of these families, by suitably modifying the definition of the soluble graph, we can construct a graph associated to $G$.

Let $\mathcal{F}$ be such a family of groups and consider the graph $\Lambda_{\mathcal{F}}(G)$ on the elements of $G$, where distinct vertices $x$ and $y$ are adjacent if and only if $\la x,y \ra$ is in $\mathcal{F}$. As before, we would like to define the related graph $\Gamma_{\mathcal{F}}(G)$ on $G \setminus I_{\mathcal{F}}(G)$, where $I_{\mathcal{F}}(G)$ is the set of isolated vertices in the complement of $\Lambda_{\mathcal{F}}(G)$. Recall that if $\mathcal{F} = \mathcal{A}$ is the class of abelian groups, then $I_{\mathcal{F}}(G) = Z(G)$ and $\Gamma_{\mathcal{F}}(G)$ is the commuting graph. Similarly, if $\mathcal{F} = \mathcal{S}$ is the class of soluble  groups, then \cite[Theorem 1.1]{gu} implies that $I_{\mathcal{F}}(G) = R(G)$ is the soluble radical and $\Gamma_{\mathcal{F}}(G)$ is the soluble graph of $G$. Notice that in both of these special cases, $I_{\mathcal{F}}(G)$ is a normal subgroup of $G$. In addition, if $\mathcal{F} = \mathcal{S}$, then $\Gamma_{\mathcal{F}}(G)$ is connected if and only if $\Gamma_{\mathcal{F}}(G/I_{\mathcal{F}}(G))$ is connected and moreover, the diameters of these two graphs are equal (see Lemma \ref{l:rad}). 

These attractive properties provide further impetus for studying the soluble graph of a finite group, compared with some of the other possibilities for $\mathcal{F}$ mentioned above. Indeed, various difficulties arise when we switch our focus to one of the other families. 

For example, suppose $\mathcal{F} = \mathcal{M}$ is the class of metabelian groups. Here it is difficult to give an efficient description of the isolated vertices $I_{\mathcal{F}}(G)$, which need not even be a subgroup of $G$ (note that $I_{\mathcal{F}}(G)$ is always a normal subset of $G$).

\begin{ex}\label{e:metabelian}
Let $p$ be a prime number and let $G$ be a Sylow $p$-subgroup of ${\rm GL}_{n}(p)$, where $n \geqs 2$. Notice that $G$ is generated by its abelian normal subgroups. In addition, note that if $N$ is an abelian normal subgroup of $G$, then $\la x,y \ra$ is metabelian for all $x \in N$ and $y \in G$. Therefore, $I_{\mathcal{M}}(G)$ is a subgroup of $G$ if and only if $I_{\mathcal{M}}(G) = G$. However, if $n$ is large enough, then $G$ contains $2$-generated subgroups that are not metabelian and thus $I_{\mathcal{M}}(G)$ is not a subgroup (indeed, any given $p$-group embeds in ${\rm GL}_{n}(p)$ for some $n$).
\end{ex}

Let us assume $I_{\mathcal{F}}(G)$ is a subgroup of $G$. The next problem is to determine whether or not the connectivity of $\Gamma_{\mathcal{F}}(G/I_{\mathcal{F}}(G))$ implies the connectivity of $\Gamma_{\mathcal{F}}(G)$. Working with the family 
$\mathcal{M}$ of metabelian groups, we present the following example. Here    $I_{\mathcal{M}}(G)$ is a subgroup of $G$ and $\Gamma_{\mathcal{M}}(G/I_{\mathcal{M}}(G))$ is connected, but $\Gamma_{\mathcal{M}}(G)$ is disconnected.

\begin{ex}\label{ex:meta}
Let $G={\rm SL}_2(3) = Q_8{:}C_3$ and observe that $I_{\mathcal{M}}(G)=Z(G) = C_2$. Then $G/I_{\mathcal{M}}(G) \cong A_4$ is metabelian and $\Gamma_{\mathcal{M}}(G/I_{\mathcal{M}}(G))$ is the null graph (on zero vertices). However $\Gamma_{\mathcal{M}}(G)$ has $22$ vertices and $5$ connected components:  one comprising the $6$ elements of order $4$ and the remainder corresponding to the $4$ elements of order $3$ or $6$ in each of the four cyclic subgroups of $G$ with order $6$.
\end{ex}

Although the previous example shows that $\Gamma_{\mathcal{M}}(G)$ is not connected, in general, we can establish the following result as an easy consequence of  Theorem \ref{t:main1}.

\begin{prop}\label{p:meta}
Let $G$ be a nontrivial finite group with $R(G)=1$. Then the metabelian graph $\Gamma_{\mathcal{M}}(G)$ is connected and its diameter is at most $2 \delta_{\mathcal{S}}(G)$.
\end{prop}

\begin{proof}
Let $x$ and $y$ be two adjacent vertices in the soluble graph $\Gamma_{\mathcal{S}}(G).$ If $N$ is a minimal normal subgroup of
$\langle x, y \rangle$ and $1\neq n \in N,$ then $\langle x, n\rangle$ and $\langle y, n\rangle$ are metabelian, so $x$ and $y$ have distance at most $2$ in $\Gamma_{\mathcal{M}}(G)$. The result now follows, noting that $\Gamma_{\mathcal{S}}(G)$ is connected by Theorem \ref{t:main1}.
\end{proof}

A natural question that arises at this point is the following.

\begin{prob}\label{p:conn}
For which families $\mathcal{F}$ of finite soluble groups is it true that $\Gamma_{\mathcal{F}}(G)$ is connected for every finite group $G\not\in \mathcal{F}$? 
\end{prob}

Notice that the nilpotent and supersoluble graphs of $G = A_4$ are equal and disconnected. Indeed this graph has $11$ vertices and $5$ connected components: one comprising the $3$ involutions, and four more consisting of an element of order $3$ and its inverse. So the answer to Problem \ref{p:conn} is negative if $\mathcal{F}$ is the family of nilpotent groups or the family of supersoluble groups. We expect it would be interesting to consider the same question for some larger families of groups, such as those with derived nilpotent subgroup, or with Fitting length at most $2$.

Even when the answer to Problem \ref{p:conn} is negative, it would be interesting to investigate the following two questions.

\begin{prob}
Let $\mathcal{F}$ be a family of finite groups and define $I_{\mathcal{F}}(G)$ and $\Gamma_{\mathcal{F}}(G)$ as above.
\begin{itemize}\addtolength{\itemsep}{0.2\baselineskip}
\item[{\rm (i)}] Does there exist an absolute constant $c$ such that if $G$ is a finite  group, then every connected component of $\Gamma_{\mathcal{F}}(G)$ has diameter at most $c$?
\item[{\rm (ii)}] Is there a positive answer to (i) if we only consider groups with 
$I_{\mathcal{F}}(G)=1$?
\end{itemize}
\end{prob}

Notice that if $\mathcal F=\mathcal{A}$ is the family of abelian groups, then $I_{\mathcal{F}}(G) = Z(G)$, $\Gamma_{\mathcal{F}}(G)$ is the commuting graph and the two previous questions have different answers. Indeed, as we recalled in Section \ref{s:intro}, if $Z(G)=1,$ then each connected component of the commuting graph of a non-abelian finite group $G$ has diameter at most $10$ (this is a theorem of Morgan and Parker \cite{MP}). However, Giudici and Parker \cite{GP} have constructed an  infinite sequence $(G_n)$ of finite $2$-groups such that the diameter of the corresponding commuting graphs tends to infinity. 

If we take $\mathcal{F} = \mathcal{N}$ to be the family of nilpotent groups, then the following stronger result holds.
  
\begin{prop}
Let $G$ be a finite non-nilpotent group. Then each connected component of the nilpotent graph $\Gamma_{\mathcal{N}}(G)$ of $G$ has diameter at most $10$.
\end{prop}  

\begin{proof}
First observe that $I_{\mathcal{N}}(G)$ coincides with the hypercentre of $G$, denoted $Z_\infty(G)$, which is the final term in the upper central series of $G$ (see \cite[Proposition 2.1]{az}). Moreover, for all $x, y \in G$, we note that $\langle x,y  \rangle$ is nilpotent if and only if $\langle xZ_\infty(G), yZ_\infty(G)\rangle$ is a nilpotent subgroup of $G/Z_\infty(G)$. This implies that there is a bijective correspondence between the connected components of $\Gamma_{\mathcal{N}}(G)$ and  
$\Gamma_{\mathcal{N}}(G/Z_\infty(G))$. Moreover, the corresponding components under this bijection have the same diameter, so we are free to assume that $Z_\infty(G)=1.$

Next observe that if $x$ and $y$ are adjacent vertices in $\Gamma_{\mathcal{N}}(G)$, then they have distance at most two in the commuting graph $\Gamma_{\mathcal{A}}(G).$ Indeed, there is a path $x \sim z \sim y$ in $\Gamma_{\mathcal{A}}(G)$ for every nontrivial element $z \in Z(\la x,y \ra)$. Therefore, $\Gamma_{\mathcal{N}}(G)$ and the commuting graph of $G$ have the same connected components and the result now follows by applying the main theorem of \cite{MP}.
\end{proof}

Finally, let $\mathcal{C}$ be the family of metacyclic groups. We conclude this section with the following result on the metacyclic graph $\Gamma_{\mathcal{C}}(G)$ of a finite simple group.

\begin{prop}\label{p:metac}
Let $G$ be a non-abelian finite simple group. Then either
\begin{itemize}\addtolength{\itemsep}{0.2\baselineskip}
\item[{\rm (i)}] The metacyclic graph $\Gamma_{\mathcal{C}}(G)$ is connected; or
\item[{\rm (ii)}] $G = {\rm L}_{2}(3^f)$ and $f \geqs 3$ is odd.
\end{itemize}
\end{prop}

\begin{proof}
First recall that any two involutions in $G$ generate a dihedral group, which is metacyclic, and thus the set of involutions in $G$ form a clique in $\Gamma_{\mathcal{C}}(G)$. In particular, there is a connected component $\O$ that contains every involution in $G$.  
	
We claim that if $p$ is a prime and $\Omega$ contains an element $x$ of order $p$, then it contains every element in $G$ whose order is divisible by $p$. To see this, suppose $y \in G$ has order divisible by $p$ and write $|y|=p^am$, where $a \geqs 1$ and $(p,m)=1$. Let $P$ be a Sylow $p$-subgroup of $G$ containing $y^m$ and a conjugate $x^g$ of $x$. Let $z$ be a nontrivial element of $Z(P)$. Since $x \in \Omega$, there exists a path 
\[
x_1 \sim x_2 \sim \cdots \sim x_n
\]
in $\Gamma_{\mathcal{C}}(G)$, where $x_1 = x$ and $x_n$ is an involution. But then 
\[
y \sim y^m \sim z \sim x^g \sim x_2^g \sim \cdots \sim x_n^g
\]
is also a path in $\Gamma_{\mathcal{C}}(G)$ and we conclude that $y \in \Omega$. In particular, $\Omega$ contains every element in $G$ of even order.

Next let $x \in G$ be a real element with $|x| \geqs 3$, which means that $x^g=x^{-1}$ for some $g \in G$. Since $g$ must have even order and $\langle x,g\rangle$ is metacyclic, it follows that $x \in \Omega$. 

Finally, let us assume $G \ne {\rm L}_2(3^f)$ with $f \geqs 3$ odd. We claim that $\Omega$ contains at least one element of order $p$ for every odd prime divisor $p$ of $|G|$. In view of our first claim above, this immediately implies that $\Gamma_{\mathcal{C}}(G)$ is connected as in (i). So it remains to establish this claim. 

Suppose the claim is false and let $p$ be the smallest prime divisor of $|G|$ such that $\Omega$ contains no element of order $p$. Then $G$ does not contain a real element of order $p$ and so the possibilities for $(G,p)$ are determined in \cite[Theorem 2.1]{DMN} by Dolfi et al. We consider each possibility in turn.

First assume $G \neq {\rm L}_2(q)$ with $q=p^f \equiv 3 \imod{4}$. Let $P$ be a Sylow $p$-subgroup of $G$ with $|P| = p^n$ and note that $P = \la u \ra$ is cyclic by \cite[Theorem 2.1]{DMN}. Now $P\not\leqs  Z(N_G(P))$ since $G$ is not $p$-nilpotent, so there exists an element $v \in N_G(P)\setminus C_G(P)$. In particular,  $vC_G(P)$ is a nontrivial element of $N_G(P)/C_G(P)\leqs {\rm Aut}(P)$, so its order divides $\varphi(p^n)=p^{n-1}(p-1)$ and is coprime to $p$ (since $P\leqs C_G(P)$). But then there is a prime $q$ dividing $(|v|,p-1)$ and so the minimality of $p$ implies that  $v$ is contained in $\Omega$. However, $\langle u, v\rangle$ is metacyclic and thus $u \sim v$ in $\Gamma_{\mathcal{C}}(G)$, which means that $u \in \Omega$ and we have reached a contradiction.

Finally, let us assume $G = {\rm L}_2(q)$ with $q=p^f \equiv 3 \imod{4}$. Then $G$ contains a subfield subgroup $H = {\rm L}_2(p)$, which in turn contains a metacyclic subgroup of order $p(p-1)/2$ (a Borel subgroup of $H$). Therefore, if $p \ne 3$ then $\Omega$ contains an element of order $p$ and the proof of the proposition is complete.
\end{proof}

\begin{rem}
Let us observe that there exist groups $G$ arising in part (ii) of Proposition \ref{p:metac} for which $\Gamma_{\mathcal{C}}(G)$ is disconnected. For example, suppose $G = {\rm L}_{2}(27)$ and let $H$ be a maximal subgroup of $G$ with order divisible by $3$. Then either $H = 3^3{:}13$ is a Borel subgroup or $H = A_4$. In particular, we deduce that 
$\Gamma_{\mathcal{C}}(G)$ has $28$ connected components, each containing $3$ elements (one such component for each Sylow $3$-subgroup of $G$), plus an additional connected component comprising the remaining $743$ elements in $G$. 
\end{rem}

\section{Some related problems}\label{s:op}

In this final section, we present some open problems on the soluble graph of a finite group that arise naturally from our work in this paper. Throughout this section, $G$ denotes a finite insoluble group with $R(G)=1$.

Our first problem concerns the sharpness of the main bound in Theorem \ref{t:main1}.

\begin{prob}\label{pr:1}
Is there a finite group with $\delta_{\mathcal{S}}(G) = 5$? 
\end{prob}

\begin{prob}\label{pr:11}
Are there infinitely many groups with $\delta_{\mathcal{S}}(G) \geqs 4$?
\end{prob}

By Theorem \ref{t:main4}, we have shown that there are infinitely many alternating groups with $\delta_{\mathcal{S}}(G) \geqs 4$, modulo the conjectured existence of  infinitely many Sophie Germain primes. A positive solution to the following conjecture would resolve Problem \ref{pr:11} (note that the condition $p \geqs 11$ is necessary since $\delta_{\mathcal{S}}(A_7) = 3$). 

\begin{con}\label{con:1}
Let $p \geqs 11$ be a prime with $p \equiv 3 \imod{4}$. Then $\delta_{\mathcal{S}}(A_p) \geqs 4$.
\end{con}

In part (ii)(d) of Theorem \ref{t:main2} we observe that $\delta_{\mathcal{S}}(G) \geqs 3$ for every non-abelian finite simple group $G$ that is not isomorphic to a classical group. For $G$ classical, we know that $\delta_{\mathcal{S}}(G) = 2$ if $G = {\rm L}_2(q)$ with $q$ even (or $q \in \{5,7\}$) or if $G \in \{{\rm L}_{3}(2), {\rm U}_4(2)\}$. Are there any additional simple groups with this property? 

\begin{prob}\label{pr:3}
Determine all the simple groups with $\delta_{\mathcal{S}}(G) = 2$.
\end{prob}

Recall part (i) of Theorem \ref{t:main2}, which states that $\delta_{\mathcal{S}}(G) \leqs 3$ if $G$ is not almost simple. If $G$ is almost simple with socle $G_0$, then we know that the same bound holds if $G \ne G_0$ and $G_0$ is an alternating or sporadic group (see Theorem \ref{t:spor2}, \eqref{e:a56} and Theorem \ref{t:sym}). This observation leads naturally to the following problem.

\begin{prob}\label{pr:2}
Do we have $\delta_{\mathcal{S}}(G) \leqs 3$ for every non-simple group $G$? 
\end{prob}

Let $G$ be an almost simple group with socle $G_0$. It is worth noting that there are examples with $\delta_{\mathcal{S}}(G) = \delta_{\mathcal{S}}(G_0) \pm 1$. For instance,
\begin{align*}
\delta_{\mathcal{S}}({\rm M}_{12}.2) & = 3 = \delta_{\mathcal{S}}({\rm M}_{12})-1 \\
\delta_{\mathcal{S}}({\rm U}_{4}(2).2) & = 3 = \delta_{\mathcal{S}}({\rm U}_{4}(2))+1.
\end{align*}
It appears that there are very few groups with $\delta_{\mathcal{S}}(G) > \delta_{\mathcal{S}}(G_0)$. Indeed, $G = {\rm U}_4(2).2$ is the only example we are aware of.

\begin{prob}\label{pr:4}
Let $G$ be an almost simple group with socle $G_0$.
\begin{itemize}\addtolength{\itemsep}{0.2\baselineskip}
\item[{\rm (i)}] Classify the almost simple groups with $\delta_{\mathcal{S}}(G) > \delta_{\mathcal{S}}(G_0)$.
\item[{\rm (ii)}] Do we always have $|\delta_{\mathcal{S}}(G) - \delta_{\mathcal{S}}(G_0)| \leqs 1$?
\end{itemize}
\end{prob}

We can extend the previous problem to more general monolithic groups as follows.

\begin{prob}\label{pr:5}
Let $G$ be a monolithic group with socle $T^k$, where $T$ is a non-abelian finite simple group and $k \geqs 2$. 
\begin{itemize}\addtolength{\itemsep}{0.2\baselineskip}
\item[{\rm (i)}] If $\delta_{\mathcal{S}}(T) \geqs 4$ then do we always have $\delta_{\mathcal{S}}(G) = 3$? 
\item[{\rm (ii)}] Do we always have $|\delta_{\mathcal{S}}(G) - \delta_{\mathcal{S}}(T)| \leqs 1$? 
\end{itemize}
\end{prob}

Note part (ii) of Problem \ref{pr:5} can only have a positive solution if Problem \ref{pr:1} has a negative answer in the sense that $\delta_{\mathcal{S}}(T) \leqs 4$ for every non-abelian finite simple group $T$.

\end{document}